\documentclass[11pt,reqno]{amsart}
\usepackage{graphicx} 
\usepackage[utf8]{inputenc}
\usepackage{amsmath,amssymb}
\usepackage{wrapfig}
\usepackage{url}
\usepackage{mathtools}
\usepackage{graphicx}
\usepackage{stmaryrd}
\usepackage{amsthm}
\usepackage{mathrsfs}
\usepackage{xcolor}
\usepackage[shortlabels]{enumitem}
\usepackage{comment}
\usepackage{relsize}
\usepackage{ dsfont }
\usepackage{hyperref}
\usepackage{stmaryrd}

\usepackage{geometry}
\newtheorem{theorem}{Theorem}[section]
\numberwithin{equation}{section}
\newtheorem{corollary}[theorem]{Corollary}
\newtheorem{prop}[theorem]{Proposition}
\newtheorem{example}[theorem]{Example}

\newtheorem{lemma}[theorem]{Lemma}
\theoremstyle{definition}
\newtheorem{obs}{Remark}
\theoremstyle{definition}
\newtheorem{definition}[theorem]{Definition}

\DeclareMathOperator{\Pt}{\mathnormal{\prescript{t}{}{P}}}

\newcommand{\supp}{\text{supp}}

\newcommand{\T}{\mathbb{T}}

\newcommand{\G}{\mathnormal{\mathbb{T}^{m}\times\mathbb{R}^n}}
\newcommand{\R}{\mathbb{R}}
\newcommand{\Z}{\mathbb{Z}}
\newcommand{\N}{\mathbb{N}}
\newcommand{\s}{\mathcal{S}(\G)}

\newcommand{\g}{\mathbb{T}^1\times\mathbb{R}}

\newcommand{\Real}{\mathrm{Re}}
\newcommand{\Imag}{\mathrm{Im}}

\renewcommand{\S}{\mathcal{S}_{\sigma,\mu}(\G)}
\renewcommand{\s}{\mathcal{S}_{\sigma,\mu}(\T^1\times\R)}
\renewcommand{\SS}{\mathcal{S}'_{\sigma,\mu}(\G)}
\renewcommand{\ss}{\mathcal{S}'_{\sigma,\mu}(\T^1\times\R)}
\newcommand{\SC}{\mathcal{S}_{\sigma,\mu,C}(\G)}
\newcommand{\Smu}{\mathscr{F}_{\mu}(\T^1\times\R)}
\newcommand{\SSmu}{\mathscr{F}'_{\mu}(\T^1\times\R)}

\newcommand{\scoef}{\mathcal{S}_{\sigma_{\scalebox{0.5}{$0$}},\mu}(\g)}

\newcommand{\sscoef}{\mathcal{S}'_{\sigma_{\scalebox{0.5}{$0$}},\mu}(\g)}

\newcommand{\sigmac}{\sigma_{\scalebox{0.5}{$0$}}}

\newcommand{\longhookrightarrow}{\lhook\joinrel\longrightarrow}

\DeclareFontFamily{U}{mathx}{\hyphenchar\font45}
\DeclareFontShape{U}{mathx}{m}{n}{
      <5> <6> <7> <8> <9> <10>
      <10.95> <12> <14.4> <17.28> <20.74> <24.88>
      mathx10
      }{}
\DeclareSymbolFont{mathx}{U}{mathx}{m}{n}
\DeclareFontSubstitution{U}{mathx}{m}{n}
\DeclareMathAccent{\widecheck}{0}{mathx}{"71}
\DeclareMathOperator{\ftil}{\widehat{\widehat{\mathnormal{f}\,}}\!\!\,}

\allowdisplaybreaks

\newcommand{\defeq}{\vcentcolon=}

\title[Global hypoellipticity on time-periodic Gelfand-Shilov spaces]{Global hypoellipticity on time-periodic Gelfand-Shilov spaces via non-discrete Fourier analysis}

\author[A. Kowacs]{Andr\'e Pedroso Kowacs}
\address{
	Universidade Federal do Paran\'{a},
	Departamento de Matem\'{a}tica,
	C.P.19096, CEP 81531-990, Curitiba, Paran\'a, Brazil
}
\email{andrekowacs@gmail.com}

\author[P. Tokoro]{Pedro Meyer Tokoro}
\address{
	Universidade Federal do Paran\'{a},
	Programa de P\'os-Gradua\c c\~ao de Matem\'{a}tica,
	C.P.19096, CEP 81531-990, Curitiba, Paran\'a, Brazil
}
\email{pedro.tokoro@ufpr.br}

\thanks{This study was financed in part by the Coordena\c c\~ao de Aperfei\c coamento de Pessoal de N\'ivel Superior - Brasil (CAPES) - Finance Code 001.}
	


\begin{document}

\begin{abstract}
In this paper, we provide a characterization of the time-periodic Gelfand–Shilov spaces, as introduced by F. de \'Avila Silva and M. Cappiello [J. Funct. Anal., 282(9):29, 2022], through the asymptotic behaviour of both the Euclidean and periodic partial Fourier transforms of their elements. As an application, we establish necessary and sufficient conditions for global regularity---within this framework---for a broad class of constant-coefficient differential operators, as well as for first-order tube-type operators.
\end{abstract}

\maketitle

\section{Introduction}

Time-periodic Gelfand-Shilov spaces were defined by F. de \'Avila Silva and M. Cappiello in \cite{AvCap2022} to study the global regularity of time-periodic evolution operators on Gelfand-Shilov spaces. More precisely, the authors study the global hypoellipticity of operators the form
\[L=D_t+c(t)P(x,D_x)\]
on $\T^1\times\R^n$, where  $\T^1=\mathbb{S}^1\simeq\R/2\pi\Z$ is the $1$-dimensional torus, $c$ is a Gevrey complex-valued function on $\T^1$, and $P(x,D_x)$ is a self-adjoint differential operator with polynomial coefficients on $\R^n$ satisfying global ellipticity conditions.

Given $\mu,\nu>0$, the Gelfand-Shilov space $\mathcal{S}_\nu^\mu(\R^n)$ is the space of all smooth functions $\varphi$ for which there exists $C>0$ such that
\[\sup_{\alpha,\beta\in\N_0^n}C^{-|\alpha|-|\beta|}\alpha!^{-\nu}\beta!^{-\mu}\sup_{x\in\R^n}|x^\alpha\partial_x^\beta \varphi(x)|<+\infty.\]

These spaces were introduced by I. Gelfand and G. E. Shilov in \cite{GelfandShilov} as an alternative setting for studying differential equations. The functions belonging to $\mathcal{S}_\nu^\mu(\R^n)$ present uniform Gevrey regularity and a certain exponential decay at infinity, and the spaces $\mathcal{S}_\mu^\nu(\R^n)$ have applications in the study of well-posedness of evolution equations, see \cite{AAC2022,AC2019} for instance. Moreover, functions on Gelfand-Shilov spaces arise naturally in applications. For instance, as shown in \cite{CGR2011}, the eigenfunctions of generalized harmonic oscillators belong to $\mathcal{S}_\nu^\mu(\R^n)$ for suitable $\mu$ and $\nu$.

The time-periodic Gelfand-Shilov space $\S$, as introduced in \cite{AvCap2022}, is the space of smooth functions which are uniformly Gevrey regular of order $\sigma$ on the periodic variables and Gelfand-Shilov behaviour of order $\mu$ on the real variables. 
More precisely, it is the space of functions $\varphi\in C^\infty(\G)$ such that
\[\sup_{\alpha,\beta\in\N_0^n}\sup_{\gamma\in \N_0^m} C^{-|\alpha+\beta|-|\gamma|}\gamma!^{-\sigma}(\alpha!\beta!)^{-\mu}\sup_{(t,x)\in\T^m\times\R^n}|x^\alpha\partial_x^\beta\partial_t^\gamma \varphi(t,x)|<+\infty,\]
for some $C>0$.

In \cite{AvCap2022}, the authors developed a discrete Fourier analysis for $\S$ by combining the partial Fourier transform on the torus with an eigenfunction expansion induced by a globally elliptic self-adjoint operator $P(x,D_x)$ on $\R^n$ as developed in \cite{GPR2011}.

Inspired by \cite{KowTmRn,Kow_chap}, we introduce a Fourier analysis for functions on time-periodic Gelfand-Shilov spaces, combining the Fourier series on the torus with the classical Euclidean  Fourier transform. More precisely, we provide a characterization for functions on $\S$ by means of the decay of its partial and total Fourier transforms, as well as a partial characterization of $\SS$.

As an application of these results, we study the global regularity for certain classes of differential operators $P$ in the general sense that
\begin{equation}\label{SmuGH}
    u\in\SS,\ Pu\in\S\ \Rightarrow\ u\in\S.
\end{equation}

 In this work, we consider operators that are not covered by the results of \cite{AvCap2022,AvCap2025,AvCapKir2024}, as linear partial differential operators with constant coefficients and tube-type first order operators. More specifically, we obtain necessary and sufficient conditions for such regularity property for constant coefficients differential operators on the cylinder $\T^1\times\R$ with no mixed terms satisfying suitable assumptions, as follows.

\begin{theorem}\label{theo_gh_cte_intro}
    Let $P$ be a differential operator on $\T^1\times\R$ with constant coefficients and no mixed terms, such that  $$P=p\left(D_x\right)+q\left(D_t\right),$$
    where $p$ and $q$ are complex polynomials and $\deg(p)=N$. Suppose that one of the following conditions hold:
   \begin{enumerate}
       \item $N\leq 1$, $\mu\geq 1/2$.
       \item $N>1$, $\mu\geq 1$ and there exist $C,R>0$ such that
       \begin{equation*}
       |p(\xi)+q(k)|\geq C|\xi|^{N-1},\quad \text{for every } k\in\Z,\ |\xi|\geq R. 
       \end{equation*}
   \end{enumerate}
   
   Then $P$ satisfies \eqref{SmuGH} for every $\sigma\geq 1$ if and only if its symbol does not vanish on $\Z\times \R$.
\end{theorem}

 Due to the non-discreteness of the variable $\xi$, the existence of a root for $p+q$ in $\Z\times \R$ allows the construction of a singular solution, a fact that does not happen when we deal with operators in the context of a discrete Fourier analysis, as in the case of the torus, compact Lie groups or when dealing with the Fourier analysis induced by certain elliptic operators in general. We also note that our results do not exhibit Diophantine approximation phenomena, as usually occurs in the discrete setting, see for instance \cite{AKM19,Ber1994,Ber1999,BCM1993,AGK2018,GW,Hounie,KMR2021} and references therein. This lack of Diophantine conditions  
 on the continuous setting is to be expected as  first evidenced by \cite{Hounie}.
 
 In the tube case, that is, operators of the form
\[P=\partial_t+c(t)\partial_x+q(t),\]
for any Gevrey regular complex-valued functions $c,q$ on $\T^1$, we characterize when the regularity property \eqref{SmuGH} holds, but allowing for certain loss of Gevrey regularity on the periodic variables, as in \cite{AvCap2022}. More precisely, when $c$ is real valued, we show that the global regularity of $P$ is reduced to study this property for an operator with constant coefficients. When of $\Imag(c)$ is not zero, then the global regularity of $P$ is equivalent to $\Imag(c)$ not changing sign, a property that is related to the famous Nirenberg-Tr\`eves condition $(\mathcal{P})$, in addition to the global regularity of an associated operator with constant coefficients.

 We can summarize our results of the tube case as follows.
 \begin{theorem}
     Let $\sigma\ge 1$, and $P$ be a differential operator on $\T^1_t\times \R_x$ given by 
     \begin{equation*}
         P=\partial_t+(a(t)+ib(t))\partial_x+q(t),
     \end{equation*}
     where $a,b,q\in\mathscr{G}^\sigma(\T^1)$ and $a,b$ are real valued. Also let $a_0,b_0,q_0$ denote the respective averages over $\T^1$. Then for every $\mu> 1$,  $P$ is $\mathscr{F}_\mu$-globally hypoelliptic (see Definition \ref{Fmu_GH_def}) if and only if one of the following holds
     \begin{enumerate}
         \item $a,b,q$ are constant and \begin{enumerate}[i)]
             \item $b\neq 0$ and $\dfrac{a}{b}\Real(q)+\Imag(q)\not\in\Z$, or
             \item $b=0$ and either $\Real(q)\neq 0$, or both   $a=\Real(q)=0$ and $\Imag(q)\not\in\Z$.
         \end{enumerate}
         \item $b\equiv0$ and either $\Real(q_0)\neq 0$, or both $a_0=\Real(q_0)=0$ and $\Imag(q_0)\not\in\Z$.
         \item $b\not\equiv0$, $b$ does not change sign and $\dfrac{a_0}{b_0}\Real(q_0)+\Imag(q_0)\not\in\Z$.
       \end{enumerate}
       Moreover, if $b\equiv0$ or $a,b,q$ are constant, then the characterization above also holds for $\mu\geq 1/2$.
 \end{theorem}

This paper is structured in the following manner:  in Section \ref{sec_prelim}, we recall the fundamental theory of time-periodic Gelfand–Shilov spaces and prove certain  topological properties of these spaces. In Section \ref{sec_Fourier} we introduce the toroidal and Euclidean partial Fourier transforms in the time-periodic Gelfand–Shilov spaces, and study how these classes of functions and distributions behave under these Fourier transforms. In Sections \ref{sec_gh1} and \ref{sec_gh2} we apply the theory and results from the previous sections to the study of certain types of global hypoellipticity of two classes of differential operators. Appendix \ref{appendix} contains a few technical lemmas used in the proofs throughout the paper.

\section{Preliminaries}\label{sec_prelim}

\subsection{Gelfand-Shilov spaces}

Given $\mu,\nu>0$, the Gelfand-Shilov space $\mathcal{S}_{\nu}^{\mu}(\R^n)$ is defined as the space of the functions $\varphi\in C^\infty(\R^n)$ such that there exists $C>0$ satisfying
\[\sup_{\alpha,\beta\in\N_0^n}C^{-|\alpha|-|\beta|}\alpha!^{-\nu}\beta!^{-\mu}\sup_{x\in\R^n}|x^\alpha\partial_x^\beta \varphi(x)|<+\infty.\]

We have also the following equivalent characterizations for Gelfand-Shilov spaces, which we refer the reader to \cite[Theorem 6.1.6 and Proposition 6.1.7]{NicRod} for a proof.

\begin{theorem}\label{gelf_equiv}
	A function $\varphi\in C^\infty(\R^n)$ belongs to $\mathcal{S}_{\nu}^{\mu}(\R^n)$ if and only if one of the following equivalent conditions hold:
	\begin{enumerate}
		\item There exist $C,L>0$ such that
		\begin{equation*}\label{gelf_2}
			\sup_{\beta\in\N_0^n}\sup_{x\in\R^n} C^{-|\beta|}\beta!^{-\mu}e^{L|x|^{1/\nu}}|\partial_x^\beta \varphi(x)|<+\infty.
		\end{equation*}
        \item There exist $C,L>0$ such that
		\begin{equation*}\label{gelf_1}
			|\varphi(x)|\leq Ce^{-L|x|^{1/\nu}}\quad\text{and}\quad |\widehat{\varphi}(\xi)|\leq Ce^{-L|\xi|^{1/\mu}}.
		\end{equation*}
	\end{enumerate}
\end{theorem}

Clearly, every Gelfand-Shilov space is contained in the Schwartz space $\mathcal{S}(\R^n)$, and  $\varphi\in\mathcal{S}_{\nu}^{\mu}(\R^n)$ if and only if $\widehat{\varphi}\in\mathcal{S}_{\mu}^{\nu}(\R^n)$. In particular, if $\mu=\nu$, the space $\mathcal{S}_\mu(\R^n)=\mathcal{S}_\mu^\mu(\R^n)$ is invariant under the action of the Fourier transform. 

Notice that the functions of $\mathcal{S}_\nu^\mu(\R^n)$ are Gevrey functions on $\R^n$ that admit exponential decay at infinity. In particular,
\[\mathscr{G}_c^\mu(\R^n)\subset \mathcal{S}^\mu_\nu(\R^n)\subset \mathscr{G}^\mu(\R^n),\]
where $\mathscr{G}^\mu(\R^n)$ is the space of Gevrey functions of order $\mu$ on $\R^n$, that is, the space of the $\varphi\in C^\infty(\R^n)$ such that, for every compact $K\subset\R^n$, there exist $C_0,C_1>0$ such that
\[\sup_{x\in K}|\partial_x^\alpha \varphi(x)|\leq C_0 C_1^{\alpha}\alpha!^\mu,\]
for all $\alpha\in\N_0^n$, and $\mathscr{G}_c^\mu(\R^n)$ is the space of compactly supported Gevrey functions on $\R^n$, which is non-trivial for $\mu> 1$. For more details, see \cite[Chapter I]{Rod_Gevrey}.

Now, for $\mu,\nu>0$ fixed, given $C>0$, consider the space $\mathcal{S}^\mu_{\nu,C}(\R^n)$ of all the smooth functions $\varphi$ such that
\[\|\varphi\|_{\mu,\nu,C}=\sup_{\alpha,\beta\in\N_0^n}\sup_{x\in\R^n}C^{-|\alpha|-|\beta|}\alpha!^{-\nu}\beta!^{-\mu}|x^\alpha\partial_x^\beta \varphi(x)|<+\infty.\]
Then, each $\mathcal{S}^\mu_{\nu,C}(\R^n)$ is a Banach space with the norm $\|\cdot\|_{\mu,\nu,C}$ and
\[\mathcal{S}^\mu_\nu(\R^n) = \bigcup_{C>0}\mathcal{S}^\mu_{\nu,C}(\R^n),\]
and we endow $\mathcal{S}^\mu_\nu(\R^n)$ with the inductive limit topology. For more details on the characterization of the Gelfand-Shilov spaces, we refer the reader to \cite{Petersson} and Section 6 of \cite{NicRod}.

In this work, we will restrict our attention to the symmetric case $\mu=\nu\geq1/2$ due to the invariance of these spaces under the Fourier transform and the fact that the Gelfand-Shilov spaces are trivial when $\mu+\nu<1$ \cite[Theorem 6.1.10]{NicRod}.

\subsection{Time-periodic Gelfand-Shilov spaces}

Here, we recall the definition of time-periodic Gelfand-Shilov spaces as introduced by F. de \'Avila Silva and M. Cappiello in \cite{AvCap2022}, as well as the fundamental facts concerning these spaces.

First, let us recall the definition of the time-periodic Schwartz $\mathcal{S}(\G)$ as established in \cite{KowTmRn,Kow_chap}.

\begin{definition}
    Let $\mathcal{S}(\G)$ be the space of all $f\in C^\infty(\G)$ such that
    \[\|f\|_{\alpha,\beta,\gamma} = \sup_{(t,x)\in \G}|x^\alpha\partial_x^\beta\partial_t^\gamma f(t,x)|<+\infty,\]
    for all $\alpha,\beta\in\mathbb{N}_0^n$ and $\gamma\in\mathbb{N}_0^m$.
\end{definition}

The seminorms $\|\cdot\|_{\alpha,\beta,\gamma}$ endow $\mathcal{S}(\G)$ with a Fr\'echet space structure. Proceeding as in \cite[Theorem 14.4]{Treves_TVS}, one can show that $\mathcal{S}(\G)$ is a Montel space.

Fix $\sigma\geq 1$ and $\mu\geq 1/2$. Given $C>0$, we say that a function $\varphi\in C^\infty(\T^m\times\R^n)$ belongs to $\SC$ if
\[\|\varphi\|_{\sigma,\mu,C}:=\sup_{\alpha,\beta\in\N_0^n}\sup_{\gamma\in \N_0^m} C^{-|\alpha+\beta|-|\gamma|}\gamma!^{-\sigma}(\alpha!\beta!)^{-\mu}\sup_{(t,x)\in\T^m\times\R^n}|x^\alpha\partial_x^\beta\partial_t^\gamma \varphi(t,x)|<+\infty.\]

Each $\SC$ is a Banach space with the norm $\|\cdot\|_{\sigma,\mu,C}$.

\begin{definition}
    Given $\sigma\geq 1$ and $\mu\geq 1/2$, define the time-periodic Gelfand-Shilov space as
\[\S:=\bigcup_{C>0}\SC,\]
endowed with the inductive limit topology.
\end{definition}

Consequently, a sequence $(\varphi_j)_{j\in\N}$ on $\S$ converges to a $\varphi\in\S$ in the sense of $\S$ if there exists $C>0$ such that $\varphi,\varphi_j\in \mathcal{S}_{\sigma,\mu,C}(\G)$ for all $j\in\N$ and $\|\varphi_j-\varphi\|_{\sigma,\mu,C}\to 0$.

\begin{obs}\label{equiv_tpgelf}
    From the equivalence stated on Theorem \ref{gelf_equiv}, the norms
    \[\|\varphi\|_{\sigma,\mu,C,L}:=\sup_{\beta\in\N_0^n}\sup_{\gamma\in \N_0^m} C^{-|\beta|-|\gamma|}\gamma!^{-\sigma}\beta!^{-\mu}\sup_{(t,x)\in\T^m\times\R^n}e^{L|x|^{1/\mu}}|\partial_x^\beta\partial_t^\gamma \varphi(t,x)|,\]
    with $C,L>0$, define an equivalent topology on $\S$.
\end{obs}

The next result gives us that $\S$ is a DFS space, that is, $\S$ is the inductive limit of a compact sequence of Banach spaces (see \cite{Komatsu1967}).

\begin{theorem}\label{comp_inc}
    If $C<C_+$, the inclusion $\mathcal{S}_{\sigma,\mu,C}(\G)\hookrightarrow \mathcal{S}_{\sigma,\mu,C_+}(\G)$ is continuous and compact. In particular, $\S$ is a DFS space.
\end{theorem}
\begin{proof}
    First, notice that, for all $C>0$, the inclusion $\SC\hookrightarrow\mathcal{S}(\G)$ is continuous. Indeed, given $f\in\SC$,  $\alpha,\beta\in\mathbb{N}_0^n$ and $\gamma\in\mathbb{N}_0^m$, we have that
    \begin{align*}
        \|f\|_{\alpha,\beta,\gamma} & = \sup_{(t,x)\in \G}|x^\alpha\partial_x^\beta\partial_t^\gamma f(t,x)|\\
        & = C^{|\alpha+\beta|+|\gamma|}(\alpha!\beta!)^{\mu}\gamma!^\sigma C^{-|\alpha+\beta|-|\gamma|}(\alpha!\beta!)^{-\mu}\gamma!^{-\sigma}\sup_{(t,x)\in \G}|x^\alpha\partial_x^\beta\partial_t^\gamma f(t,x)|\\
        & \leq  (C^{|\alpha+\beta|+|\gamma|}(\alpha!\beta!)^{\mu}\gamma!^\sigma)\|f\|_{\sigma,\mu,C}.
    \end{align*}

    Also, it is clear that the inclusion map $\mathcal{S}_{\sigma,\mu,C}(\G)\hookrightarrow \mathcal{S}_{\sigma,\mu,C_+}(\G)$ is continuous. It remains to show its compactness. Let $(f_j)_{j\in\N}$ be a bounded sequence in $\SC$, that is, there exists $h>0$ such that $\|f_j\|_{\sigma,\mu,C}\leq h$, for all $j\in\mathbb{N}$. 
    
    Since the inclusion $\SC\hookrightarrow\mathcal{S}(\G)$ is continuous and by the fact that $\mathcal{S}(\G)$ is a Montel space, it follows that $(f_j)_{j\in\N}$ admits a subsequence $(f_{j_k})_{k\in\N}$ which converges to some $f\in\mathcal{S}(\G)$ in the sense of $\mathcal{S}(\G)$, that is:
    \[\sup_{(t,x)\in\G}|x^\alpha\partial_x^\beta\partial_t^\gamma f_{j_k}(t,x)| \to \sup_{(t,x)\in \G}|x^\alpha\partial_x^\beta\partial_t^\gamma f(t,x)|,\]
    for every $\alpha,\beta\in\mathbb{N}_0^n$ and $\gamma\in\mathbb{N}_0^m$. 
    Since
    \[\sup_{(t,x)\in \G}C^{-|\alpha+\beta|-|\gamma|}(\alpha!\beta!)^{-\mu}\gamma!^{-\sigma}|x^\alpha\partial_x^\beta\partial_t^\gamma f_{j_k}(t,x)| \leq h,\]
    for every $\alpha,\beta\in\mathbb{N}_0^n$, $\gamma\in\mathbb{N}_0^m$, and $k\in\mathbb{N}$, the same inequality also holds for $f$, which gives us that $f\in\SC$.
    
    Let us show that $f_{j_k}\to f$ in $\mathcal{S}_{\sigma,\mu,C_+}(\G)$. Given $\varepsilon>0$, let $K>0$ be such that $(C_+/C)^{-K}<\varepsilon/2h$. Writing $C_+=C(C_+/C)$, for  $|\alpha+\beta|+|\gamma|\geq K$, we obtain that
    \[\sup_{(t,x)\in\G}C_+^{-|\alpha+\beta|-|\gamma|}(\alpha!\beta!)^{-\mu}\gamma!^{-\sigma}|x^\alpha\partial_x^\beta\partial_t^\gamma (f_{j_k}-f)(t,x)| \leq \frac{\varepsilon}{2h}\|f_{j_k}-f\|_{\sigma,\mu,C} < \varepsilon,\]
    for every $k\in\N$. Also, since $f_{j_k}\to f$ in $\mathcal{S}(\G)$ we can choose $k_0\in\mathbb{N}$ large enough so that
    \[\max_{|\alpha+\beta|+|\gamma|<K}\sup_{(t,x)\in\G}|x^\alpha\partial_x^\beta\partial_t^\gamma (f_{j_k}-f)(t,x)|<\frac{\varepsilon}{C_0},\]
    for every $k\geq k_0$, where
    \[C_0 = \displaystyle\max_{|\alpha+\beta|+|\gamma|<K}\{C_+^{-|\alpha+\beta|-|\gamma|}(\alpha!\beta!)^{-\mu}\gamma!^{-\sigma}\}.\]
    Therefore
    \[\sup_{(t,x)\in\G}C_+^{-|\alpha+\beta|-|\gamma|}(\alpha!\beta!)^{-\mu}\gamma!^{-\sigma}|x^\alpha\partial_x^\beta\partial_t^\gamma (f_{j_k}-f)(t,x)| \leq C_0\frac{\varepsilon}{C_0} = \varepsilon,\]
    for every $k\geq k_0$ and $|\alpha+\beta|+|\gamma|<K$.
    
    We conclude that $\|f_{j_k}-f\|_{\sigma,\mu,C_+}<\varepsilon$ for every $k\geq k_0$. Since $\varepsilon>0$ is arbitrary, we have that $f_{j_k}\to f$ in $\mathcal{S}_{\sigma,\mu,C_+}(\G)$, which gives us the compactness of the inclusion map.    
\end{proof}

\begin{definition}
    Given $\sigma\geq 1$ and $\mu\geq 1/2$, let $\SS$ be the set of of all linear functionals $u : \S \to \mathbb{C}$ such that
\[\varphi_j\to \varphi\ \text{in}\ \S\ \Rightarrow\ \langle u,\varphi_j\rangle\to\langle u,\varphi\rangle\ \text{in}\ \mathbb{C},\]
that is, the set of sequentially continuous linear functionals on $\S$. We will refer to its elements as tempered ultradistributions.
\end{definition}

Since the space $\S$ is DFS (in particular, a Montel DF space), we have that sequential continuity is equivalent to continuity, see \cite{Webb}. As a consequence of the previous considerations, we obtain the following characterization:

\begin{prop}\label{char_dist}
    A linear functional $u:\S\to\mathbb{C}$ belongs to $\SS$ if and only if, for each $C_1>0$, there exists a constant $C_0>0$ such that 
    \[|\langle u,\varphi\rangle| \leq C_0\sup_{\alpha,\beta\in\N_0^n}\sup_{\gamma\in\N_0^m} C_1^{-|\alpha+\beta|-|\gamma|}\gamma!^{-\sigma}(\alpha!\beta!)^{-\mu}\sup_{(t,x)\in\G}|x^\alpha\partial_x^\beta\partial_t^\gamma\varphi(t,x)|,\]
    for all $\varphi\in\S$.
\end{prop}
Next, recall that for $\sigma\geq 1$, the Gevrey space on the torus $\mathscr{G}^\sigma(\T^m)$ is defined as the space of smooth functions $\psi:\T^m\to\mathbb{C}$ such that
\begin{equation}\label{ineq_Gevrey}
   \sup_{t\in\T^m}|\partial_t^{\gamma}\psi(t)|<C_0 C_1^{\gamma} \gamma!^{\sigma},
\end{equation}
for some $C_0,C_1>0$ and every $\gamma\in \N_0^m$, where $C_0$ and $C_1$ do not depend on $\gamma$.

The next proposition shows that $\S$ is a  $\mathscr{G}^\sigma(\T^m)$-module.

\begin{prop}\label{prop_closed_under_product}
    If $\psi\in\mathscr{G}^{\sigma}(\T^m)$ and $f\in\S$, then $\psi f\in \S$. 
\end{prop}
\begin{proof}
    If $\mathscr{G}^{\sigma}(\T^m)$ and $f\in\S$, then there exist positive constants $C_{0,\psi}$, $C_{1,\psi}$, $C_{0,f}$ and $C_{1,f}$ such that
    \[\sup_{t\in\T^m}|\partial_t^\gamma\psi(t)|\leq C_{0,\psi} C_{1,\psi}^{|\gamma|}\gamma!^\sigma,\ \forall\gamma\in\N_0^m,\]
    and
    \[\sup_{x\in\R^n}|x^\alpha\partial_x^\beta\partial_t^\gamma f(t,x)|\leq C_{0,f}C_{1,f}^{|\alpha+\beta|+|\gamma|}(\alpha!\beta!)^\mu\gamma!^\sigma,\ \forall \alpha,\beta\in\N_0^n,\ \gamma\in\N_0^m.\]

    By the inclusions in Theorem \ref{comp_inc}, we can assume without loss of generality that all these constants are greater than $1$. Then, for every $\alpha,\beta\in\N_0^n,\gamma\in\N_0^m$, by the Leibniz formula, we have that
    \begin{align*}
    \left|x^\alpha\partial_x^{\beta}\partial_t^\gamma[\psi(t)f(t,x)]\right| & \leq \sum_{\ell\leq \gamma}\binom{\gamma}{\ell}|\partial_t^{\ell} \psi(t)||x^{\alpha}\partial_x^{\beta}\partial_t^{\gamma-\ell}f(t,x)|\\
    &\leq \sum_{\ell\leq \gamma}\binom{\gamma}{\ell}C_{0,\psi}C_{1,\psi}^{|\ell|}\ell!^\sigma C_{0,f}C_{1,f}^{|\alpha+\beta|+|\gamma-\ell|}(\gamma-\ell)!^{\sigma}(\alpha!\beta!)^{\mu}\\
     &\leq C_{0,\psi} C_{0,f} 2^{|\gamma|}C_{1,\psi}^{|\gamma|} C_{1,f}^{|\alpha+\beta|+|\gamma|}\gamma!^{\sigma}(\alpha!\beta!)^{\mu}\\
    &\leq C_0C_1^{|\alpha+\beta|+|\gamma|}\gamma!^{\sigma}(\alpha!\beta!)^{\mu},
    \end{align*}
     for every $(t,x)\in\T^m\times\R^n$, where $C_0=C_{0,\psi} C_{0,f}$ and $C_1=2C_{1,\psi}C_{1,f}$. Therefore $\psi f\in\S$.
\end{proof}
\section{Fourier analysis}\label{sec_Fourier}

	In this section we introduce the partial and total Fourier transforms for time-periodic Gelfand-Shilov functions and ultradistributions. The construction is made by taking a partial Fourier series on the periodic variables and a partial Fourier transform on the real variables, following the ideas introduced by the first author in \cite{KowTmRn}. Moreover, we give a characterization of functions on these spaces by means of the decay of the corresponding Fourier coefficients. Additionally, we obtain some properties for certain classes of tempered ultradistributions. For the rest of this section, we shall consider arbitrary $\sigma\geq 1$ and $\mu\geq 1/2$, unless specified otherwise.

    \subsection{Fourier analysis for time-periodic Gelfand-Shilov functions}
     
    We define the partial Fourier transforms of a function $f\in L^1(\G)$ by:
    \[\mathcal{F}_{\R^n}(f)(t,\xi)=\widehat{f}(t,\xi)\defeq \int_{\R^n}f(t,x)e^{-i\xi\cdot x }\,\mathrm{d}x,\]
    for a.e. $(t,\xi)\in \G$, and
    \[\mathcal{F}_{\T^m}(f)(k,x)=\widehat{f}(k,x)\defeq \frac{1}{(2\pi)^m}\int_{[0,2\pi]^m}f(t,x)e^{-ik\cdot t}\,\mathrm{d}t,\]
     for all $k\in\Z^m$ and a.e. $x\in\R^n$. Also define its mixed partial Fourier transform by
     \begin{align*}
         \mathcal{F}_{\T^m}(\mathcal{F}_{\R^n}(f))(k,\xi) = \ftil(k,\xi)&\defeq\frac{1}{(2\pi)^m}\int_{[0,2\pi]^m}\int_{\R^n}f(t,x)e^{-i(k\cdot t+\xi\cdot x )}\,\mathrm{d}x\,\mathrm{d}t,
    \end{align*}
     for all $k\in\Z^m$ and a.e. $\xi\in\R^n$.
     
     Note that, by Fubini's Theorem, these expressions are all well defined and, moreover, $\widehat{f}(\cdot,\xi)\in L^1(\T^m)$ for every $\xi\in\R^n$ and $\widehat{f}(k,\cdot),\ftil(k,\cdot)\in L^1(\R^n)$ for every $k\in\Z^m$. 

    \begin{lemma}\label{lema331pet}
        Let $f\in\S$. The following statements are equivalent:
        \begin{itemize}
            \item[(i)] There exist $C_0,C_1,\varepsilon>0$ such that
            \[\sup_{x\in\R^n}\sup_{k\in\Z^m}|x^\alpha\partial_x^\beta \widehat{f}(k,x)|\leq C_0C_1^{|\alpha+\beta|}(\alpha!\beta!)^\mu e^{-\varepsilon|k|^{1/\sigma}},\quad\forall \alpha,\beta\in\N_0^n;\]
            
            \item[(ii)] There exist $C,h>0$ such that
            \[\sup_{x\in\R^n}\sup_{k\in\Z^m\setminus\{0\}}|x^\alpha\partial_x^\beta\widehat{f}(k,x)|\leq C^{|\alpha+\beta|}(\alpha!\beta!)^\mu h^{N+1}N!|k|^{-N/\sigma},\quad\forall N\in\N,\ \alpha,\beta\in\N_0^n;\]
            
            \item[(iii)] There exist $C,h>0$ such that
            \[\sup_{x\in\R^n}\sup_{k\in\Z^m\setminus\{0\}}|x^\alpha\partial_x^\beta\widehat{f}(k,x)|\leq C^{|\alpha+\beta|}(\alpha!\beta!)^\mu h(hN)^N|k|^{-N/\sigma},\quad\forall N\in\N_0,\ \alpha,\beta\in\N_0^n.\]
        \end{itemize}
    \end{lemma}
    \begin{proof}
    The proof follows from the definition and by arguments similar to those found in the proof of \cite[Lemma 1.6.2]{Rod_Gevrey}.
    \end{proof}

    \begin{lemma}\label{lema_petr}
        Given $C,N>0$ and $\sigma\geq 1$, let $M$ be the smallest positive integer such that $M\geq N/\sigma$. Then, there exists $C'>0$ independent on $N$ such that
        \[C^{M+1}M^{\sigma M}\leq C'(C'N)^N.\]
    \end{lemma}
    \begin{proof}
        \cite[Lemma 1.6.3]{Rod_Gevrey}.
    \end{proof}
    
    \begin{lemma}\label{lema332petr}
        Let $f\in\S$. Given $N>0$,  suppose that there exist $C_0,C_1,L,h>0$ such that
        \[|x^\alpha\partial_t^{\gamma}\partial_x^\beta f(t,x)|\leq C_0C_1^{|\alpha+\beta|}(\alpha!\beta!)^{\mu} h^{|\gamma|}M^{\sigma|\gamma|},\]
        for all $\alpha,\beta\in\N_0^n$, $\gamma\in\N_0^m$, $(t,x)\in\G$, and $|\gamma|\leq M$, where $M$ is the smallest positive integer such that $M\geq N/\sigma$. Then, for all $\alpha,\beta\in\N_0^n$, $k\in\Z^m\setminus\{0\}$, and $x\in\R^n$:
        \[|x^\alpha\partial_x^\beta\widehat{f}(k,x)|\leq C_0C_1^{|\alpha+\beta|}(\alpha!\beta!)^\mu C'(C'N)^{N}|k|^{-N/\sigma},\]
        for some constant $C'>0$ which does not depend on $N$.
    \end{lemma}
    \begin{proof}
        First, notice that
        \[x^\alpha\partial_x^\beta \widehat{\partial_t^\gamma f}(k,x) = x^\alpha (ik)^\gamma 
 \partial_x^\beta\widehat{f}(k,x).\]
        Therefore if $|\gamma|=M$ we have that
        \[|k^\gamma x^\alpha\partial_x^{\beta}\widehat{f}(k,x)|\leq \frac{1}{(2\pi)^{m}}\int_{\T^m}|x^\alpha\partial_x^\beta\partial_t^\gamma f(t,x)|\,\mathrm{d}t \leq C_0C_1^{|\alpha+\beta|}(\alpha!\beta!)^\mu h^{M}M^{\sigma M},\]
        for all $\alpha,\beta\in\N_0^n$, $k\in\Z^m$ and $x\in\R^n$. If $k\neq 0$, this implies that
        \begin{align*}
        |k|^{N/\sigma}|x^\alpha\partial_x^\beta \widehat{f}(k,x)| & \leq |k|^{M}|x^\alpha\partial_x^\beta \widehat{f}(k,x)|\ \leq\  \sum_{|\gamma|=M}\dfrac{M!}{\gamma!}|k^\gamma||x^\alpha\partial_x^\beta\widehat{f}(k,x)|\\
        & \leq\ \sum_{|\gamma|=M}\dfrac{M!}{\gamma!}C_0C_1^{|\alpha+\beta|}(\alpha!\beta!)^\mu h^M M^{\sigma M}\\
        & = m^M C_0C_1^{|\alpha+\beta|}(\alpha!\beta!)^\mu h^M M^{\sigma M},
        \end{align*}
        for all $\alpha,\beta\in\N_0^n$, $k\in\Z^m$ and $x\in\R^n$. Therefore
        \[|k|^{N/\sigma}|x^\alpha\partial_x^\beta\widehat{f}(k,x)|\leq \tilde{C}^{M+1}M^{\sigma M}C_0C_1^{|\alpha+\beta|}(\alpha!\beta!)^\mu,\]
        where $\tilde{C}=\max\{1,mh\}$. Applying Lemma \ref{lema_petr} to the inequality above, we have that there exists $C'>0$ independent on $N$ such that
        \[|k|^{N/\sigma}|x^\alpha\partial_x^\beta\widehat{f}(k,x)|\leq C_0C_1^{|\alpha+\beta|}(\alpha!\beta!)^\mu C'(C'N)^N,\]
        for all $\alpha,\beta\in\N_0^n$, $k\in\Z^m$ and $x\in\R^n$, which finishes the proof.
    \end{proof}

    \begin{prop}\label{proppartialdecayrn}
    A function $f\in C^\infty(\G)\cap L^1(\T^m\times\R^n)$ belongs to $\S$ if and only if there exist $C_0,C_1,\varepsilon>0$ such that, for all $\alpha,\beta\in\N_0^n$, we have that
    
    \begin{equation}{\label{ineqpartialtorus}}
	\sup_{(k,x)\in\Z^m\times\R^n}e^{\varepsilon |k|^{1/\sigma}}|x^\alpha\partial_x^\beta\widehat{f}(k,x)|\leq C_0C_1^{|\alpha+\beta|}(\alpha!\beta!)^{\mu} .
    \end{equation}
    \end{prop}
    \begin{proof}
    First, suppose that $f\in\S$. Then given $N>0$, we have that
    \[|x^\alpha \partial_x^\beta\partial_t^\gamma f(t,x)|\leq C_0C_1^{|\alpha+\beta|}C_1^{|\gamma|}(\alpha!\beta!)^\mu\gamma!^\sigma \leq C_0 C_1^{|\alpha+\beta|}(\alpha!\beta!)^\mu C_1^{M}M^{\sigma M},\]
    for every $\alpha,\beta\in\N_0^n$, $\gamma\in\N_0^m$, $(t,x)\in\G$, with $|\gamma|=M$, where $M$ is the smallest positive integer such that $M\geq N/\sigma$. By Lemma \ref{lema332petr}, we have that condition {\it (iii)} from Lemma \ref{lema331pet} holds, implying that inequality \eqref{ineqpartialtorus} also holds.

    On the other hand, suppose that \eqref{ineqpartialtorus} holds. The previous estimates imply that the series
    \[g(t,x)\defeq\sum_{k\in\Z^m}\widehat{f}(k,x)e^{ik\cdot t}\]
    converges absolutely and uniformly on $\G$. Moreover, the series 
    \[\sum_{k\in\Z^m}\partial_x^\beta\partial_t^\gamma\left(\widehat{f}(k,x)e^{ik\cdot t}\right)\]
    also converges absolutely and uniformly on $\G$ for every $\beta\in\N_0^m$ and $\gamma\in\N_0^m$, so we can exchange the order of summation and differentiation and the series above corresponds to $\partial_x^\beta\partial_t^\gamma g(t,x)$, from which we conclude that $g\in C^\infty(\G)$. Finally, given $\alpha,\beta\in\N_0^n$ and $\gamma\in\N_0^m$, we have that
    \begin{align*}
    \sup_{(t,x)\in\T^m\times \R^n}|x^\alpha \partial_x^\beta \partial_t^\gamma g(t,x)|&=  \sup_{(t,x)\in\T^m\times \R^n} \left|x^\alpha\partial_x^\beta\partial_t^\gamma\sum_{k\in\Z^m}\widehat{f}(k,x)e^{ik\cdot t}\right|\\ &\leq \sup_{x\in\R^n} \sum_{k\in\Z^m}|k|^{|\gamma|}|x^\alpha \partial_x^\beta\widehat{f}(k,x)|\\
    & \leq C_0C_1^{|\alpha+\beta|}(\alpha!\beta!)^\mu\sum_{k\in\Z^n} |k|^{|\gamma|}e^{-\varepsilon|k|^{1/\sigma}} <\ +\infty.
    \end{align*}
    
    Finally, by the Fourier inversion theorem, it follows that the series above coincides with $f$, which finishes the proof.
    \end{proof}

    \begin{obs}
        In view of the previous theorem, given a sequence $(g(k,x))_{k\in\Z^m}$ satisfying \eqref{ineqpartialtorus}, the series
        \[\sum_{k\in\Z^m}g(k,x)e^{ikt}\]
        defines a function $g\in\S$. Also, if $f\in\S$, we have that
        \[f(t,x)=\sum_{k\in\Z^m}\widehat{f}(k,x)e^{ikt}.\]
    \end{obs}

    \begin{lemma}\label{lemma_poly}
        Fix $\nu\in\N_0^n$ and let $f\in\S$. Then, there exist $C_{0,\nu},C_{1,\nu}\geq 1$ such that, for all $\alpha,\beta\in\N_0^n$, and $\gamma\in\N_0^m$,
        \[\sup_{(t,x)\in\G}|x^\nu x^\alpha\partial_x^\beta\partial_t^\gamma f(t,x)|\leq C_{0,\nu} C_{1,\nu}^{|\alpha+\beta|+|\gamma|}(\alpha!\beta!)^\mu\gamma!^\sigma.\]
    \end{lemma}
    \begin{proof}
       Since $f\in\S$, there exist $C_0,C_1\geq 1$ such that, for all $\alpha,\beta\in\N_0^n$, $\ \gamma\in\N_0^m$ and $(t,x)\in\G$ we have that
        \[|x^\nu x^\alpha\partial_x^\beta\partial_t^\gamma f(t,x)|\leq C_0C_1^{|\alpha+\beta|+|\gamma|}C_1^{|\nu|}(\alpha+\nu)!^\mu\beta!^\mu\gamma!^\sigma.\]
       
       Therefore
        \begin{align*}
        |x^\nu x^\alpha\partial_x^\beta\partial_t^\gamma f(t,x)| & \leq C_0C_1^{|\alpha+\beta|+|\gamma|}C_1^{|\nu|}(2^{|\alpha|+|\nu|}\alpha!\nu!)^\mu\beta!^\mu\gamma!^\sigma\\
        & \leq  C_0(2^{\mu|\nu|}C_1^{|\nu|}\nu!^\mu)(2^\mu C_1)^{|\alpha+\beta|+|\gamma|}(\alpha!\beta!)^\mu\gamma!^\sigma.
        \end{align*}
        
        Taking $C_{0,\nu} = C_02^{\mu|\nu|}C_1^{|\nu|}\nu!^\mu$ and $C_{1,\nu}=2^\mu C_1$, the proof is complete.
    \end{proof}

    \begin{prop}\label{propdecaypartialxi}
		A function $f\in C^\infty(\G)\cap L^1(\T^m\times\R^n)$ belongs to $\S$ if and only if there exist $C_0,C_1>0$ such that, for all $\alpha,\beta\in\N_0^n$ and $\gamma\in\N_0^m$, we have that
	\begin{equation}\label{ineqpartialr}
        \sup_{(t,x)\in \T^m\times\R^n} |\xi^\alpha\partial_\xi^\beta\partial_t^\gamma\widehat{f}(t,\xi)|\leq C_0C_1^{|\alpha+\beta|+|\gamma|}(\alpha!\beta!)^\mu\gamma!^\sigma.
	\end{equation}
	\end{prop}
    \begin{proof}
    First, suppose that $f\in\S$. By Lemma \ref{lemma_poly}, we have that
    \begin{align*}
        |\xi^\alpha\partial_\xi^\beta\partial_t^\gamma\widehat{f}(t,\xi)| & = \left|\int_{\R^n}e^{-i\xi\cdot x}x^\beta\partial_x^\alpha\partial_t^\gamma f(t,x)\,\mathrm{d}x\right|\\
        & \leq \int_{\R^n}(1+|x|)^{-N}(1+|x|)^{N}|x^\beta\partial_x^\alpha\partial_t^\gamma f(t,x)|\,\mathrm{d}x\\
        & \leq \sup_{x\in\R^n}(1+|x|) ^{N}|x^\beta\partial_x^\alpha\partial_t^\gamma f(t,x)| \int_{\R^n}(1+|x|)^{-N}\,\mathrm{d}x\\
        & \leq C_0C_1^{|\alpha+\beta|+|\gamma|}(\alpha!\beta!)^\mu\gamma!^\sigma,
    \end{align*}
    where $N\in\N$ satisfies $N>n$ so that $\int_{\R^n}(1+|x|)^{-N}\,\mathrm{d}x<+\infty$, and $C_{0}=(n+1)^NC_{0,N}, C_1=(n+1)^NC_{1,N}$ are given by Lemma \ref{lemma_poly}, since $(1+|x|)^{N}$ is a polynomial on $|x_1|,\dots,|x_n|$.

    Conversely, suppose that \eqref{ineqpartialr} holds. By Lebesgue's Dominated Convergence Theorem and the Fourier inversion theorem, we have that
    \[\frac{1}{(2\pi)^n}\int_{\R^n}e^{i\xi\cdot x}\widehat{f}(t,\xi)\,\mathrm{d}\xi\]
    is a smooth function and coincides with $f$. Moreover,
    \begin{align*}
        |x^\alpha\partial_x^\beta\partial_t^\gamma f(t,x)| & = \left|\frac{1}{(2\pi)^n}\int_{\R^n}e^{i\xi\cdot x}\xi^\beta\partial_\xi^\alpha\partial_t^\gamma \widehat{f}(t,\xi)\,\mathrm{d}\xi\right|.
    \end{align*}
    
    Following the same arguments of the proof of the first implication, it follows that $f$ belongs to $\S$.
    \end{proof}

    \begin{corollary}\label{coro_gelf_transf} 
        Let $f\in L^1(\T^m\times\R^n)$. Then $f\in\S$ if and only if $\mathcal{F}_{\R^n}f\in\S$.
    \end{corollary}

    \begin{prop}\label{proppartialdecaymix}
    A function $f\in C^\infty(\G)\cap L^1(\G)$ belongs to $\S$ if and only if there exist $C_0,C_1,\varepsilon>0$ such that, for all $\alpha,\beta\in\N_0^n$,
    \begin{equation}\label{ineqmix}
    \sup_{(k,\xi)\in\Z^m\times\R^n} e^{\varepsilon|k|^{1/\sigma}}|\xi^\alpha\partial_{\xi}^\beta\ftil(k,\xi)|\leq C_0C_1^{|\alpha+\beta|}(\alpha!\beta!)^\mu .
    \end{equation}
    \end{prop}
    \begin{proof}
    The statement is a direct consequence of Proposition \ref{proppartialdecayrn} and Corollary \ref{coro_gelf_transf}.
    \end{proof}

    \begin{obs}
        From now, we will denote $\ftil(k,\xi)$ simply by $\widehat{f}(k,\xi)$ for simplicity.
    \end{obs}

    \subsection{Fourier analysis for ultradistributions}

    Now, let $u\in \SS$. Define its partial Fourier transforms by
    $$\langle\mathcal{F}_{\R^n}(u)(t,\xi),f(t,\xi)\rangle = \langle u(t,\xi),\mathcal{F}_{\R^n}(f)(t,\xi)\rangle,$$
    for every $ \,f\in\SS$, and
    $$\mathcal{F}_{\T^m}(u)(k,\cdot)\in \mathcal{S}'_\mu(\R^n),$$
    given by 
    $$\langle\mathcal{F}_{\T^m}(u)(k,x),f(x)\rangle = \langle u(t,x),e^{ik\cdot t}f(x)\rangle,$$
    for every $f\in\mathcal{S}_\mu(\R^n)$ and $k\in\Z^m$.
    We also define its mixed partial Fourier transform by
    $$\mathcal{F}_{\T^m}(\mathcal{F}_{\R^n}(u))(k,\cdot)\in \mathcal{S}'_\mu(\R^n),$$
    given by
    $$\left\langle \mathcal{F}_{\T^m}(\mathcal{F}_{\R^n}(u))(k,\xi),f(\xi)\right\rangle = \langle u(t,\xi),e^{-ik\cdot t}\mathcal{F}_{\R^n}(f)(\xi)\rangle,$$
    for every $f\in\mathcal{S}_\mu(\R^n)$ and $k\in\Z^m$.
    
    Lastly, for $u\in\SS$ we also define $$\langle\mathcal{F}_{\R^n}^{-1}u(t,x),f(t,x)\rangle = \langle u(t,x),\mathcal{F}_{\R^n}^{-1}f(t,x)\rangle,$$
    where 
    \begin{equation*}
        \mathcal{F}^{-1}_{\R^n}(f)(t,x)\defeq \frac{1}{(2\pi)^n}\int_{\R^n}f(t,\xi)e^{i\xi\cdot x }\,\mathrm{d}\xi
    \end{equation*}
    denotes the partial inverse Fourier transform for $f\in \S$.
    By Corollary \ref{coro_gelf_transf} and the continuity of the partial Fourier transform it follows that $\mathcal{F}_{\R^n}u,\mathcal{F}_{\R^n}^{-1}u\in\SS$, for any $u\in \SS$.

    \begin{prop}\label{proppartialdecaytorusdistrib}
    Let $u\in \SS$. Then for every $C_1,\varepsilon>0$, there exists $\tilde C_0=\tilde C_0(\varepsilon,C_1)>0$ such that
\begin{equation}\label{ineqpartialtorusdistrib}
    |\langle \widehat{u}(k,x),f(x)\rangle|\leq \tilde C_0\|f\|_{\mu,C_1}e^{\varepsilon|k|^{1/\sigma}},
    \end{equation}
    for every $f\in \mathcal{S}_\mu(\R^n)$ and $k\in\Z^m$, where
    \begin{equation}\label{norm_GS}
    \|f\|_{\mu,C_1}=\sup_{\alpha,\beta\in\N_0^n}C_1^{-|\alpha+\beta|}(\alpha!\beta!)^{-\mu}\sup_{x\in\R^n}|x^\alpha\partial_x^\beta f(x)|.
    \end{equation}
    \end{prop}
    \begin{proof}
        Given $f\in\mathcal{S}_\mu(\R^n)$, notice that $(t,x)\mapsto e^{ik\cdot t}f(x)\in\S$, for all $k\in\Z^m$. By Proposition \ref{char_dist} for every $\tilde C_1>0$ there exists $\tilde C_0=\tilde C_0(\tilde C_1)>0$ such that
        \begin{align*}
            & |\langle \widehat{u}(k,x),f(x)\rangle| \\
            & = |\langle u,e^{ik\cdot t}f(x)\rangle| \leq \tilde C_0\sup_{\alpha,\beta\in\N_0^n}\sup_{\gamma\in\N_0^m}\tilde C_1^{-|\alpha+\beta|}\tilde C_1^{-|\gamma|}\gamma!^{-\sigma}(\alpha!\beta!)^{-\mu}\sup_{t\in\T^m}|\partial_t^\gamma e^{ik\cdot t}|\sup_{x\in\R^n}|x^\alpha\partial_x^\beta f(x)|\\
            &\leq \tilde C_0\sup_{\alpha,\beta\in\N_0^n}\sup_{\gamma\in\N_0^m} \left(\frac{{C_1}}{\tilde C_1}\right)^{|\alpha+\beta|}C_1^{-|\alpha+\beta|}\tilde C_1^{-|\gamma|}(\alpha!\beta!)^{-\mu}(|\gamma|!)^{-\sigma}|k|^{|\gamma|} \sup_{x\in\R^n}|x^\alpha\partial_x^\beta f(x)| \\
            &\leq \tilde C_0\sup_{\alpha,\beta\in\N_0^n}\sup_{\gamma\in\N_0^m} \left(\frac{{C_1}}{\tilde C_1}\right)^{|\alpha+\beta|}C_1^{-|\alpha+\beta|}(\alpha!\beta!)^{-\mu}\left(\frac{(\sigma/\varepsilon)^{\sigma}}{\tilde C_1}\right)^{|\gamma|} e^{\varepsilon|k|^{1/\sigma}}\sup_{x\in\R^n}|x^\alpha\partial_x^\beta f(x)|,
        \end{align*}
        for all $k\in\Z^m$, where in the last line we applied Lemma \ref{lemma-exponential}. By choosing $\tilde C_1\geq \max\{(\sigma/\varepsilon)^\sigma,C_1\}$, we obtain
        \begin{align*}
            |\langle \widehat{u}(k,x),f(x)\rangle| \leq  \tilde C_0\sup_{\alpha,\beta\in\N_0^n}C_1^{-|\alpha+\beta|}(\alpha!\beta!)^{-\mu} \sup_{x\in\R^n}|x^\alpha\partial_x^\beta f(x)|e^{\varepsilon|k|^{1/\sigma}},
            \end{align*}
            for every $k\in\Z^m$, as claimed.
    \end{proof}

    \begin{lemma}\label{lemma_unifbound}
        Let $(u_j)_{j\in\N}$ be a sequence in $\SS$ such that, for each function $\varphi\in\S$, $(\langle u_j,\varphi\rangle)_{j\in\N}$ is a Cauchy sequence in $\mathbb{C}$. Then, there exists a unique $u\in\SS$ such that $u=\displaystyle\lim_{j\to\infty}u_j$.
    \end{lemma}
    \begin{proof}
        See \cite[Lemma 2.1]{KowTmRn} and \cite[Theorem 2.9]{AvCap2022}.
    \end{proof}

    \begin{prop}\label{inverse-sequence-distrib}
        Let $(u(k,\cdot))_{k\in\Z^m}$ be a sequence in $\mathcal{S}'_\mu(\R^n)$ which satisfies \eqref{ineqpartialtorusdistrib}. Then, there exists a unique $\widecheck u\in\SS$ such that
        \[\langle \widecheck u,f\rangle = (2\pi)^m\sum_{k\in\Z^m}\langle u(k,x),\widehat{f}(-k,x)\rangle=\sum_{k\in\Z^m}\int_{\T^m}\langle u(k,x),f(t,x)\rangle e^{ik\cdot t}\,\mathrm{d}t,\]
        for all $f\in\S$. In this case, we define $\mathcal{F}_{\T^m}^{-1}u=\widecheck{u}$.
    \end{prop}
    \begin{proof}
        For each $j\in\N$, consider the  $S_j\in\SS$ given by
        \begin{align*}
    \langle S_j(t,x),f(t,x)\rangle\ & = \sum_{|k|\leq j}\langle  u(k,x)e^{ik\cdot t},f(t,x)\rangle\\
    & = \sum_{|k|\leq j}\int_{\T^m}\langle u(k,x),f(t,x)\rangle e^{ik\cdot t}\,\mathrm{d}t  \\
    & = (2\pi)^m\sum_{|k|\leq j}\langle u(k,x),\widehat{f}(-k,x)\rangle,
    \end{align*}
    for every $f\in\S$. For each $k\in\Z^m$,  Proposition \ref{proppartialdecaytorusdistrib} implies that for all $C,\varepsilon>0$, there exists $\tilde C_0=\tilde C_0(\varepsilon,C_1)>0$ such that
    \begin{align*}
    |\langle {u}(k,x),\widehat{f}(-k,x)\rangle|&\leq  \tilde C_0\|\widehat{f}(-k,\cdot)\|_{\mu,C}e^{\varepsilon|k|^{1/\sigma}},
    \end{align*}
    for every $f\in \S$, where $\|\cdot\|_{\mu,C}$ is given by \eqref{norm_GS}. Therefore, given $j,\ell\in\N$, we have that 
    \begin{align*}
    (2\pi)^{-m}|\langle S_{j+\ell},f \rangle - \langle S_{j},f\rangle|&\leq \sum_{|k|>j}^{j+\ell} |\langle {u}(k,x),\widehat{f}(-k,x)\rangle|\leq \sum_{|k|>j}^{j+\ell}\tilde C_0\|\widehat{f}(-k,\cdot)\|_{\mu,C}e^{\varepsilon|k|^{1/\sigma}}\\
    &\leq \sum_{|k|>j}^{\infty} \tilde C_0\|\widehat{f}(-k,\cdot)\|_{\mu,C}e^{\varepsilon|k|^{1/\sigma}}.
    \end{align*}
   Since $f\in\S$, by Proposition \ref{proppartialdecayrn} there exist $C_f,\varepsilon_0,C_0>0$ such that
    \[\|\widehat{f}(k,\cdot)\|_{\mu,C_f}\leq C_0 e^{-\varepsilon_0|k|^{1/\sigma}},\]
for all $k\in\Z^m$.

    Taking $\varepsilon=\varepsilon_0/2$, we have that
    \[|\langle S_{j+\ell},f \rangle - \langle S_{j},f\rangle|\leq (2\pi)^m\sum_{|k|>j}^{\infty} C'e^{-\frac{\varepsilon_0}{2}|k|^{1/\sigma}}<+\infty,\]
    for $C'=\tilde{C}_0(C_f,\varepsilon_0/2)C_0$, which implies that
    \[\lim_{j\to\infty}|\langle S_{j+\ell},f \rangle - \langle S_{j},f\rangle| = 0\]
    due to the convergence of the series. Therefore, $(\langle S_j,f\rangle)_{j\in\N}$ is a Cauchy sequence in $\mathbb{C}$ for all $f\in\S$. It follows from Lemma \ref{lemma_unifbound} that
    \[\widecheck u = (2\pi)^m\sum_{k\in\Z^m}u(k,x)e^{ik\cdot t} = \lim_{j\to\infty}S_j\in\SS,\] 
    which finishes the proof.
    \end{proof}

    \begin{prop}\label{proppartialdecayrndistribinv}
        Consider a function $u:\G\to\mathbb{C}$ satisfying the following property: for every $\varepsilon>0$, there is $C_\varepsilon>0$ such that
        \begin{equation}\label{ineq_u_exp}
            |u(t,\xi)|\leq C_\varepsilon e^{\varepsilon|\xi|^{1/\mu}},
        \end{equation}
        for all $(t,x)\in\G$. Then, there exists $\widecheck u\in\bigcap_{\sigma\geq 1} \SS$ such that, for each $f\in\S$, we have that
        \begin{align*}
         \langle \widecheck{u},f\rangle &=\frac{1}{(2\pi)^n}\int_{\T^m} \int_{\R^n}\int_{\R^n}u(t,\xi)e^{i\xi\cdot x }f(t,x)\,\mathrm{d}\xi \,\mathrm{d}x\,\mathrm{d}t\\
         &=\frac{1}{(2\pi)^n}\int_{\T^m} \int_{\R^n} u(t,\xi)\widehat{f}(t,-\xi)\,\mathrm{d}\xi \,\mathrm{d}t.
      \end{align*}
      
      Moreover, $\mathcal{F}^{-1}_{\R^n}({u}) = \widecheck{u}.$
    \end{prop}
    \begin{proof}
        Let $f\in\S$, for some $\sigma\geq 1$. By Remark \ref{equiv_tpgelf} and Corollary \ref{coro_gelf_transf}, there exist $C,L>0$ such that
        \[|\widehat f(t,\xi)|\leq Ce^{-L|\xi|^{1/\mu}},\quad\forall (t,\xi)\in\G.\]

       	By choosing $\varepsilon=L/2$ in inequality \eqref{ineq_u_exp}, we see that the integrals above are finite. Therefore the map $f\mapsto\langle\widecheck u,f\rangle$ is well-defined and linear. Now, let $(f_j)_{j\in\N}$ be a sequence in $\S$ converging to $0$ in $\S$, for some $\sigma\geq 1$. By the continuity of Fourier transform, we have that $(\mathcal{F}_{\R^n}{f_j})_{j\in\N}\equiv (\widehat{f_j})_{j\in\N}$ also converges to $0$ in $\S$. Then, in view of Remark \ref{equiv_tpgelf}, there exist $C,L>0$ such that
        \[\|\widehat{f_j}\|_{\sigma,\mu,C,L}\to 0,\ j\to\infty.\]
        Once again considering $\varepsilon=L/2$ in inequality \eqref{ineq_u_exp}, we have that
        \begin{align*}
            |\langle \widecheck{u},f_j\rangle | & \leq (2\pi)^{-n}\int_{\T^m}\int_{\R^n}|u(t,\xi)\widehat{f_j}(t,-\xi)|\,\mathrm{d}\xi\,\mathrm{d}t\\
            & \leq {(2\pi)^{-n} C_{L/2}}\int_{\T^m}\int_{\R^n} e^{-\frac{L}{2}|\xi|^{1/\mu}}e^{L|\xi|^{1/\mu}}|\widehat{f_j}(t,-\xi)|\,\mathrm{d}\xi\,\mathrm{d}t\\
            & \leq C'\|\widehat{f_j}\|_{\sigma,\mu,C,L}\to 0,
        \end{align*}
        as $j\to\infty$, where
        \[0<C'={(2\pi)^{-n+m}C_{L/2}C}\int_{\R^n}e^{-\frac{L}{2}|\xi|^{1/\mu}}\,\mathrm{d}\xi <+\infty.\]
        Therefore we conclude that $\widecheck u\in\SS$.
    \end{proof}

    Finally, from the previous results, we obtain:

    \begin{corollary}
        Let $(u(k,\xi))_{k\in\Z^m}$ be a sequence of complex-valued functions on $\R^n$ such that, for every $\varepsilon>0$, there exists $C_\varepsilon>0$ satisfying
        \[|u(k,\xi)|\leq C_\varepsilon e^{\varepsilon|k|^{1/\sigma}}\!e^{\varepsilon|\xi|^{1/\mu}}\]
        for every $(k,\xi)\in\Z^m\times\R^n$. Then $\mathcal{F}^{-1}_{\R^n}(\mathcal{F}^{-1}_{\T^m}(u))$ is well defined and belongs to $\SS$. Moreover, $\mathcal{F}^{}_{\T^m}(\mathcal{F}^{}_{\R^n}(\mathcal{F}^{-1}_{\R^n}(\mathcal{F}^{-1}_{\T^m}(u)))) = u$ in the sense of $\SS$.
    \end{corollary}

    \section{Global hypoellipticity of operators with constant coefficients}\label{sec_gh1}

From now, we apply the Fourier analysis developed in the previous section to obtain a characterization for the global regularity of certain classes of differential operators on $\T^1\times\R$. 

We will focus in a type of global regularity commonly known as global hypoellipticity. Following \cite{AvCap2022}, in this setting this concept is  defined as follows.

\begin{definition}
    Let $\sigma\geq 1$ and $\mu\geq 1/2$. We say that a differential operator $P$ on $\T^1\times\R$ is $\mathcal{S}_{\sigma,\mu}$-globally hypoelliptic if
    \[u\in\mathcal{S}'_{\sigma,\mu}(\T^1\times\R),\ Pu\in\s\ \Rightarrow\ u\in\s.\]
\end{definition}

In this section we provide a characterization for the $\mathcal{S}_{\sigma,\mu}$-global hypoellipticity of a class of differential operators with constant coefficients.  From now on we will consider $\sigma\geq 1$ and $\mu\geq 1/2$ arbitrary.

\begin{theorem}\label{theo_gh_necessary}
Let $P$ be a differential operator on $\T^1\times\R$ with constant coefficients. If $P$ is $\mathcal{S}_{\sigma,\mu}$-globally hypoelliptic, then its symbol does not vanish on $\Z\times\R$.

\end{theorem}
\begin{proof}
First note that $P$ can be written as
\begin{equation*}
    P=p(D_t,D_x),
\end{equation*}
where $D=i^{-1}\partial$ and $p\in\mathbb{C}[x,y]$ is a complex polynomial on two variables, which also corresponds to its symbol.

Suppose that $p(k,\xi)=0$ for some $(k_0,\xi_0)\in\Z\times \R$, and consider the function $u\in C^\infty(\g)$ given by
    \[u(t,x)=e^{i(k_0t+\xi_0x)},\]
   for every $(t,x)\in\T^1\times\R$. Notice that $u\in\ss\setminus\s$ and $Pu=p(k_0,\xi_0)u=0$. Therefore, $P$ is not $\mathcal{S}_{\sigma,\mu}$-globally hypoelliptic.
\end{proof}

In order to obtain sufficient conditions for $\mathcal{S}_{\sigma,\mu}$-global hypoellipticity, we make some further assumptions on their symbols. First we consider the case of operators of order at most $1$ in the variable $x$.

\begin{theorem}\label{theo_gh_cte_Nleq1}
    Let $P$ be a differential operator on $\T^1\times\R$ with constant coefficients and no mixed terms, such that  $$P=p\left(D_x\right)+q\left(D_t\right),$$
   where $p$ and $q$ are complex polynomials and $\deg(p)\leq 1$. Then $P$ is $\mathcal{S}_{\sigma,\mu}$-globally hypoelliptic if and only if its symbol does not vanish on $\Z\times \R$.
\end{theorem}
\begin{proof}
     By Theorem \ref{theo_gh_necessary}, it is enough to prove the sufficiency of the condition above.
 
  Hence suppose that the set 
   \[\mathcal{Z}_P\defeq\{(k,\xi)\in\Z\times\R\,:\,p(\xi)+q(k)=0\}\]
   is empty.

    We claim that there exists $K_0> 0$ such that
    \begin{equation}\label{ineq_symbol_delta}
        |p(\xi)+q(k)|\geq K_0,
    \end{equation}
    for every $(k,\xi)\in\Z\times \R$.
    
    If $\deg(p)<1$, then $p(\xi)\equiv c\in\mathbb{C}$ is constant.  Consider the polynomial $\tilde q(k)\defeq c+q(k)$.
    Then  $\tilde q(k)\neq 0$ for every $k\in\Z$ by our assumption. Since $\tilde q$ is a closed map, we have that
    \begin{equation*}
        |p(\xi)+q(k)|\geq \inf_{k\in\Z}|\tilde q(k)|>0,
    \end{equation*}
    for every $(k,\xi)\in\Z\times \R$. If $\deg(p)=1$, without loss of generality we may assume that $p(\xi)=\xi+c_0$, for some $c_0\in\mathbb{C}$. Therefore
    \begin{equation*}
        |p(\xi)+q(k)|^2=|\xi+c_0+q(k)|^2= |\xi +\Real(c_0)+\Real(q(k))|^2+|\Imag(c_0)+\Imag(q(k))|^2,
    \end{equation*}
    for every $(k,\xi)\in\Z\times \R$. Note that $\mathcal{Z}_P=\varnothing$ implies $\Imag(c_0)\not\in -\Imag(q)(\Z)$, for otherwise 
    \begin{equation*}
        (k_0,-\Real(c_0)-\Real(q(k_0)))\in \mathcal{Z}_P,
    \end{equation*}
    where $k_0\in\Z$ satisfies $-\Imag(q)(k_0)=\Imag(c_0)$. Then since $\Imag(q)$ is a closed map, we have that 
\begin{equation*}
    |p(\xi)+q(k)|\geq \frac{1}{2}\operatorname{dist}(\Imag(c_0),\Imag(q)(\Z))>0,
\end{equation*}
for every $(k,\xi)\in\Z\times \R$. This concludes the proof of inequality \eqref{ineq_symbol_delta}.

    Therefore,  if $u\in\ss$ is such that $Pu=f\in\s$, by comparing Fourier coefficients we have that
    \[\widehat{u}(k,\xi)=\dfrac{\widehat{f}(k,\xi)}{p(\xi)+q(k)},\quad (k,\xi)\in\Z\times\R.\]

    Then, given $\alpha,\beta\in\N_0$,  by Lemma \ref{derivative-reciprocal} we have that 
\begin{align}
        & |\xi^\alpha\partial_\xi^\beta\widehat{u}(k,\xi)| =\left|\xi^\alpha\sum_{\ell=0}^\beta \binom{\beta}{\ell}\partial_\xi^{\ell}\left(\frac{1}{p(\xi)+q(k)}\right)\partial_\xi^{\beta-\ell}\widehat{f}(k,\xi)\right|\notag \\
        &=\left| \frac{\xi^\alpha\partial_\xi^\beta\widehat{f}(k,\xi)}{p(\xi)+q(k)}+\delta_{\deg(p),1}\xi^\alpha\sum_{\ell=1}^\beta \binom{\beta}{\ell}\partial_\xi^{\beta-\ell}\widehat{f}(k,\xi) (-1)^{\ell}\frac{1}{(p(\xi)+q(k))^{\ell+1}}\right|\notag\\
        &\leq K_0^{-1}|\xi^\alpha\partial_\xi^\beta\widehat{f}(k,\xi)|+\sum_{\ell=1}^\beta \binom{\beta}{\ell}|\xi^\alpha\partial_\xi^{\beta-\ell}\widehat{f}(k,\xi)|K_0^{-(\ell+1)} \notag,
        \end{align}
    for every $(k,\xi)\in\Z\times \R$, where $\delta_{\deg(p),1}=1$ if $\deg(p)=1$ and $0$ otherwise. 
    
    Since $f\in\s$, there exist $C_0,C_1\geq 1$ and $1\geq \varepsilon>0$ such that
    \begin{equation}\label{eq_estimates_f}     |\xi^\alpha\partial_\xi^\beta\widehat{f}(k,\xi)|\leq C_0C_1^{\alpha+\beta}(\alpha!\beta!)^\mu e^{-\varepsilon|k|^{1/\sigma}},
    \end{equation}  
    for every $\alpha,\beta\in\N_0$, $(k,\xi)\in\Z\times\R$. Therefore, for $K=\max\{K_0^{-1},1\}$, we have that 
\begin{align*}
      |\xi^\alpha\partial_\xi^\beta\widehat{u}(k,\xi)| &\leq K_0^{-1}C_0C_1^{\alpha+\beta}(\alpha!\beta!)^\mu e^{-\varepsilon|k|^{1/\sigma}}+\sum_{\ell=1}^\beta \binom{\beta}{\ell}C_0C_1^{\alpha+\beta-\ell}(\alpha!(\beta-\ell)!)^\mu e^{-\varepsilon|k|^{1/\sigma}}K_0^{-(\ell+1)}\\
      &\leq 2^\beta K^{\beta}C_0C_1^{\alpha+\beta}(\alpha!\beta!)^\mu e^{-\varepsilon|k|^{1/\sigma}}\leq \tilde C_0 \tilde C_1^{\alpha+\beta}(\alpha!\beta!)^\mu e^{-\varepsilon|k|^{1/\sigma}},
\end{align*}
for every $(k,\xi)\in\Z\times \R$, where $\tilde C_0=C_0$ and $\tilde C = 2KC_1$. 
 Hence, by Proposition \ref{proppartialdecaymix}, we have that $u\in\s$, which implies that $P$ is $\mathcal{S}_{\sigma,\mu}$-globally hypoelliptic.
\end{proof}

In particular, for first-order operators, we have the following:

\begin{prop}\label{dtaibdx}
    Let $c_1,c_2,c_3\in\mathbb{C}$. Then, the operator
    \[P=c_1\partial_x+c_2\partial_t+c_3\]
    is $\mathcal{S}_{\sigma,\mu}$-globally hypoelliptic if and only if
    \[\mathcal{Z}=\{(k,\xi)\in\Z\times\R\,:\, c_1\xi+c_2k-ic_3=0\}=\varnothing.\]
    In particular, if $c_1,c_2,c_3\in\R$ and $c_3\neq 0$, then $P$ is $\mathcal{S}_{\sigma,\mu}$-globally hypoelliptic.
\end{prop}
\begin{proof}
   Note that $P$ is $\mathcal{S}_{\sigma,\mu}$-globally hypoelliptic if and only if $i^{-1}P$ is $\mathcal{S}_{\sigma,\mu}$-globally hypoelliptic. Therefore the claim follows from Theorem \ref{theo_gh_cte_Nleq1}.
\end{proof}

\begin{corollary}\label{coro_cte}
    Let $a,b\in\R$ and $c\in\mathbb{C}$. Then, the operator
    \[P=\partial_t+(a+ib)\partial_x+c\]
    is $\mathcal{S}_{\sigma,\mu}$-globally hypoelliptic if and only if one of the following conditions hold:
    \begin{enumerate}
        \item[(i)] $b\neq 0$ and $\dfrac{a}{b}\Real(c)+\Imag(c)\notin\mathbb{Z}$;
        \item[(ii)] $b=0$ and either $\Real(c)\neq 0$, or both $a=\Real(c)=0$ and $\operatorname{Im}(c)\notin\Z$.
    \end{enumerate}
\end{corollary}
\begin{proof}
    By Proposition \ref{dtaibdx}, it is enough to show that $\mathcal{Z}=\varnothing$ if and only if {\it (i)} or {\it (ii)} holds. First, note that $(k,\xi)\in\mathcal{Z}_P$ if and only if
    \[k+a\xi+\Imag(c)=0\quad\text{and}\quad b\xi-\Real(c)=0.\]
    In these conditions, if $(k,\xi)\in\mathcal{Z}_P$ and $b\neq 0$, we must have $\xi=\dfrac{\Real(c)}{b}$. Substituting this value in the first equality, we obtain
    \[k=-\left(\dfrac{a}{b}\Real(c)+\Imag(c)\right).\]
    
    Therefore we have that if $b\neq 0$, $\mathcal{Z}=\varnothing$ if and only if $\dfrac{a}{b}\Real(c)+\Imag(c)\not\in\Z$. On the other hand, if $b=0$ and $\Real(c)\neq 0$, then $\mathcal{Z}=\varnothing$. In the case where $b=0$ and $\Real(c)=0$, if $a\neq 0$ we can choose
    \[\xi_k=-\dfrac{k+\Imag(c)}{a},\]
    for any $k\in\Z$,  so that $(k,\xi)\in \mathcal{Z}\neq\varnothing$. However, if $a=0$, then $(k,\xi)\in\mathcal{Z}_P$ if and only if $k=-\Imag(c)\in\Z$. 
\end{proof}

\begin{corollary}\label{gh_const}
    Let $a,b\in\R$ and $c\in\mathbb{C}$. Then, the operator
    \[P=\partial_x+(a+ib)\partial_t+c\]
    is $\mathcal{S}_{\sigma,\mu}$-globally hypoelliptic if and only if one of the following conditions hold:
    \begin{enumerate}
        \item[(i)] $b\neq 0$ and $\dfrac{\Real(c)}{b}\notin\mathbb{Z}$;
        \item[(ii)] $b=0$ and $\Real(c)\neq 0$.
    \end{enumerate}
\end{corollary}
\begin{proof}
    By Proposition \ref{dtaibdx}, it is enough to show that $\mathcal{Z}=\varnothing$ if and only if {\it (i)} or {\it (ii)} holds.
   Note that $(k,\xi)\in\mathcal{Z}_P$ if and only if 
    \[ak+\xi+\Imag(c)=0\quad\text{and}\quad bk-\Real(c)=0.\]
    
    If $b\neq 0$, we have that $\dfrac{\Real(c)}{b}\notin\mathbb{Z}$ implies $\mathcal{Z}=\varnothing$. On the other hand, $\dfrac{\Real(c)}{b}\in\mathbb{Z}$ implies that
    \[\left(\frac{\Real(c)}{b},-\left(\frac{\Real(c)}{b}+\Imag(c)\right)\right)\in\mathcal{Z},\]
    hence $\mathcal{Z}\neq \varnothing$. If $b=0$, then $\mathcal{Z}=\varnothing$ if and only if $\Real(c)\neq 0$, since  $ak+\xi+\Imag(c)=0$ for $\xi=-\Imag(c)-ak$ and any $k\in\Z$. 
\end{proof}

\begin{obs}\label{not_sigma}
    Notice that in the results above, the conditions for $\mathcal{S}_{\sigma,\mu}$-global hypoellipticity do not depend on $\sigma\geq 1$ nor $\mu\geq 1/2$.
\end{obs}

For operators of any order on the variable $x$, some further assumptions are necessary. From the proof of the previous theorem, note that the non-vanishing on $\Z\times \R$ condition for the symbol implies a positive lower bound for its absolute value. This is not true in general when $\deg(p)>1$. Indeed, following the notation of Theorem \ref{theo_gh_cte_Nleq1}, for $p(\xi)=-ix^2+x+i$ and $q(k)=k^2-k$ we have that $p(\xi)+q(k)\neq 0$ for every $(k,\xi)\in\R^2$, but
\begin{equation*}
    p(\sqrt{k^2+1})+q(k)=\sqrt{k^2+1}-k\to 0,
\end{equation*}
as $k\to \infty$. Moreover, due to the growth at infinity of the derivatives of $p(\xi)$ when $\deg(p)>1$, we in fact need to assume a growth condition for the symbol and restrict to the case $\mu\geq 1$, as follows.

\begin{theorem}\label{theo_gh_cte_N>1}
    Let $P$ be a differential operator on $\T^1\times\R$ with constant coefficients and no mixed terms, such that  $$P=p\left(D_x\right)+q\left(D_t\right),$$
   where $p$ and $q$ are complex polynomials,  $\deg(p)=N>1$ and there exist $C,R>0$ such that
       \begin{equation}\label{lower_bound_theorem}
           |p(\xi)+q(k)|\geq C|\xi|^{N-1},\quad \text{for every }k\in\Z,\ |\xi|\geq R.
       \end{equation}
       Then $P$ is $\mathcal{S}_{\sigma,\mu}$-globally hypoelliptic for $\mu\geq 1$ if and only if its symbol does not vanish on $\Z\times \R$.
\end{theorem}

\begin{proof}
  By Theorem \ref{theo_gh_necessary}, it is enough to prove the sufficiency of the condition above.
 
  Hence suppose that the set 
   \[\mathcal{Z}_P\defeq\{(k,\xi)\in\Z\times\R\,:\,p(\xi)+q(k)=0\}\]
   is empty.
   
Note that by \eqref{lower_bound_theorem}, there exist $K_0>0$ and $R>0$ such that
    \begin{equation*}
        |p(\xi)+q(k)|\geq K_0|\xi|^{N-1},
    \end{equation*}
    for every $(k,\xi)\in \Z\times \R$ such that $|\xi|\geq R$.

    On the other hand, since $p$ is continuous, $p(\{\xi\in\R:|\xi|\leq R\})$ is compact in $\mathbb{C}$ and since $q(\Z)\subset \mathbb{C}$ is closed, we have that $\delta_{ R}\defeq\operatorname{dist}(p(\{\xi\in\R:|\xi|\leq R\}),-q(\Z))>0$, so that
    \begin{equation}\label{delta_R}
        |p(\xi)+q(k)|\geq \delta_{R},\quad 
    \end{equation}
    for every  $|\xi|\leq  R$. Therefore taking $K_1=(1+R)^{-(N-1)}\min\{K_0 R^{N-1},{\delta_R}\}$, we have that 
    \begin{equation}\label{eq_symbol_decay_cte}
        |p(\xi)+q(k)|\geq K_1(1+|\xi|)^{N-1},
    \end{equation}
    for every $(k,\xi)\in \Z\times \R$. Therefore,  if $u\in\ss$ is such that $Pu=f\in\s$, by comparing Fourier coefficients we have that
    \[\widehat{u}(k,\xi)=\dfrac{\widehat{f}(k,\xi)}{p(\xi)+q(k)},\quad (k,\xi)\in\Z\times\R.\]

    Then, given $\alpha,\beta\in\N_0$,  by Lemma \ref{derivative-reciprocal}, for $\beta>0$ we have that 
\begin{align}
        & |\xi^\alpha\partial_\xi^\beta\widehat{u}(k,\xi)| =\left|\xi^\alpha\sum_{\ell=0}^\beta \binom{\beta}{\ell}\partial_\xi^{\ell}\left(\frac{1}{p(\xi)+q(k)}\right)\partial_\xi^{\beta-\ell}\widehat{f}(k,\xi)\right|\notag \\
        &=\left| \frac{\xi^\alpha\partial_\xi^\beta\widehat{f}(k,\xi)}{p(\xi)+q(k)}+\xi^\alpha\sum_{\ell=1}^\beta \binom{\beta}{\ell}\partial_\xi^{\beta-\ell}\widehat{f}(k,\xi) \sum_{\ell'=1}^{\ell}(-1)^{\ell'}\binom{\ell+1}{\ell'+1}\frac{\partial_\xi^{\ell}[(p(\xi)+q(k))^{\ell'}]}{(p(\xi)+q(k))^{\ell'+1}}\right|\notag\\
        &\leq K_1^{-1}(1+|\xi|)^{-(N-1)}|\xi^\alpha\partial_\xi^\beta\widehat{f}(k,\xi)|+\sum_{\ell=1}^\beta \binom{\beta}{\ell}|\xi^\alpha\partial_\xi^{\beta-\ell}\widehat{f}(k,\xi) |\sum_{\ell'=1}^{\ell}\binom{\ell+1}{\ell'+1}\frac{1}{|p(\xi)+q(k)|^{\ell'+1}}\notag\\
        &\times\sum_{\substack{|\eta|=\ell}}\binom{\ell!}{\eta_1,\dots,\eta_{\ell'}} |p(\xi)+q(k)|^{\ell'-|\supp(\eta)|}\prod_{\substack{1\leq i\leq \ell'\\\eta_i\neq 0 }}|\partial_\xi^{\eta_i}p(\xi)|,\label{ineq_u_cte}
        \end{align}
where $\supp(\eta)=\{1\leq i\leq \ell':\eta_i\neq 0\}$, for $\eta\in\N_0^{\ell'}$, since 
\begin{align*}
    \partial_\xi^{\eta_i} [p(\xi)+q(k)]=\begin{cases}
        \partial_\xi^\eta p(\xi)&\text{if }\eta_i\neq 0,\\
        p(\xi)+q(k)&\text{if }\eta_i=0.
    \end{cases}
\end{align*}

 Also, since $\deg(p)=N$, there exists $K\geq 1$ such that     
    \begin{align*}
        |\partial_\xi^\ell p(\xi)|
        \leq 
            K(1+|\xi|)^{N-\ell},
    \end{align*}
     for every $\ell\in\N_0$ and $\xi\in\R$. Therefore
        \begin{align*}     
        &\frac{1}{|p(\xi)+q(k)|^{\ell'+1}}\sum_{\substack{|\eta|=\ell}}\binom{\ell!}{\eta_1,\dots,\eta_{\ell'}} |p(\xi)+q(k)|^{\ell'-|\supp(\eta)|}\prod_{\substack{1\leq i\leq \ell'\\\eta_i\neq 0 }}|\partial_\xi^{\eta_i}p(\xi)|\\
        &= \sum_{|\eta|=\ell}\binom{\ell!}{\eta_1,\dots,\eta_{\ell'}} |p(\xi)+q(k)|^{-(1+|\supp(\eta)|)}K^{|\supp(\eta)|}(1+|\xi|)^{N\cdot |\supp(\eta)|-\ell}\\
        &\leq\sum_{|\eta|=\ell}\binom{\ell!}{\eta_1,\dots,\eta_{\ell'}} K_1^{-(1+|\supp(\eta)|)}K^{\ell'}(1+|\xi|)^{-(N-1)(1+|\supp(\eta)|)+N\cdot |\supp(\eta)|-\ell}\\
      &\leq\sum_{|\eta|=\ell}\binom{\ell!}{\eta_1,\dots,\eta_{\ell'}}K_2^{\ell+1} K^\ell(1+|\xi|)^{|\supp(\eta)|-\ell-(N-1)}\\
         &\leq K_2^{\ell+1} K^\ell\sum_{|\eta|=\ell}\binom{\ell!}{\eta_1,\dots,\eta_{\ell'}}\leq K_2^{\ell+1} K^\ell e^\ell \ell!,
\end{align*}
where $K_2=\max\{1,K_1^{-1}\}$ and we have used the fact that
\[\sum_{|\eta|=\ell}\binom{\ell!}{\eta_1,\dots,\eta_{\ell'}} = (\ell')^\ell \leq \ell^\ell \leq e^\ell\ell!.\]

Also, since $f\in\s$, there exist $C_0,C_1\geq 1$ and $1\geq \varepsilon>0$ such that
    \begin{equation}\label{eq_estimates_f_2}     |\xi^\alpha\partial_\xi^\beta\widehat{f}(k,\xi)|\leq C_0C_1^{\alpha+\beta}(\alpha!\beta!)^\mu e^{-\varepsilon|k|^{1/\sigma}},
    \end{equation}  
    for every $\alpha,\beta\in\N_0$, $(k,\xi)\in\Z\times\R$.
    
Applying these estimates to \eqref{ineq_u_cte}, we have that
\begin{align*}
     |\xi^\alpha\partial_\xi^\beta\widehat{u}(k,\xi)| &=K_1^{-1}C_0C_1^{\alpha+\beta}(\alpha!\beta!)^\mu e^{-\varepsilon|k|^{1/\sigma}}\\
    &\phantom{=}+\sum_{\ell=1}^\beta \binom{\beta}{\ell}C_0C_1^{\alpha+\beta-\ell}(\alpha!(\beta-\ell)!)^\mu e^{-\varepsilon|k|^{1/\sigma}}\sum_{\ell'=1}^{\ell}\binom{\ell+1}{\ell'+1}K_2^{\ell+1} K^\ell e^\ell \ell!\\
    &\leq K_1^{-1}C_0C_1^{\alpha+\beta}(\alpha!\beta!)^\mu e^{-\varepsilon|k|^{1/\sigma}}\\
    &\phantom{=}+\sum_{\ell=1}^\beta \binom{\beta}{\ell}C_0C_1^{\alpha+\beta-\ell}(\alpha!(\beta-\ell)!)^\mu K_2^{\beta+1} K^\beta e^\beta \ell!^\mu e^{-\varepsilon|k|^{1/\sigma}}\sum_{\ell'=1}^{\ell}\binom{\ell+1}{\ell'+1}\\
    &\leq K_1^{-1}C_0C_1^{\alpha+\beta}(\alpha!\beta!)^\mu e^{-\varepsilon|k|^{1/\sigma}}\\
    &\phantom{=}+C_0C_1^{\alpha+\beta}(\alpha!\beta!)^\mu K_2^{\beta+1} K^\beta e^\beta e^{-\varepsilon|k|^{1/\sigma}}2^\beta  (3^\beta-1)\\
    &\leq \tilde C_0\tilde C_1^{\alpha+\beta}(\alpha!\beta!)^\mu e^{-\varepsilon|k|^{1/\sigma}},
\end{align*}
for every $(k,\xi)\in\Z\times \R$, where $\tilde C_0= C_0K_2$, $\tilde C_1= 6eC_1KK_2$, and we used the fact that $\ell!\leq \ell!^\mu$, since $\mu\geq 1$, and the following estimate:
\begin{equation*}
    \sum_{\ell=1}^\beta\sum_{\ell'=1}^{\ell}\binom{\ell+1}{\ell'+1}\leq  \sum_{\ell=1}^\beta\sum_{\ell'=0}^{\ell}\binom{\ell}{\ell'}\frac{\ell+1}{\ell'+1}\leq \beta+1\sum_{\ell=1}^\beta\sum_{\ell'=0}^{\ell}\binom{\ell}{\ell'}\leq 2^\beta  (3^\beta-1) \leq 6^\beta.
\end{equation*}
     
    An analogous argument yields similar estimates for $\beta=0$.
    Hence, by Proposition \ref{proppartialdecaymix}, we have that $u\in\s$, which implies that $P$ is $\mathcal{S}_{\sigma,\mu}$-globally hypoelliptic.
\end{proof}

\begin{obs}
    A sufficient condition so that  \eqref{lower_bound_theorem} is satisfied is that the \L ojasiewicz exponent at infinity (see \cite{zbMATH05257541}) of $\Real(p+q)$ or $\Imag(p+q)$ is greater than $N-1$. Proposition 2.4 in \cite{zbMATH05789579} provides a formula for calculating such exponent via Puiseux expansions, which can be computed with the Newton polygon of these polynomials.
\end{obs}

\begin{example}
    Consider a differential operator  given by
    \begin{equation*}
        P=p(\partial_x)+q(\partial_t),
    \end{equation*}
    where $p$ is a complex polynomial, $q$ is a real polynomial and $\deg(\Imag(p))\geq \deg(p)-1$. Then for $\deg(p)=N>1$, since $\Imag(p(\xi))\sim |\xi|^{N-1}$ for large $|\xi|$, by Theorem \ref{theo_gh_cte_N>1} the operator $P$ is $\mathcal{S}_{\sigma,\mu}$-globally hypoelliptic for $\mu\geq 1$ if and only if $p+q$ does not vanish on $\Z\times \R$. For instance, the operator
    \begin{equation*}
\partial_x^{\ell+1}+i\partial_x^{\ell}+\partial_t^j+c,
    \end{equation*}
    where $\ell,j\in\N_0$, $\ell\geq2$ and $c\in\R$, is $\mathcal{S}_{\sigma,\mu}$-globally hypoelliptic for $\mu\geq 1$ if and only if $c\not\in\Z$.
\end{example}

\begin{corollary}
 Let $P$ be a differential operator on $\T^1\times\R$ with constant coefficients and no mixed terms, such that  $$P=p\left(D_x\right)+q\left(D_t\right),$$
   where $p$ and $q$ are real polynomials and $\deg(p)=N>1$. Then $P$ is 
    $\mathcal{S}_{\sigma,\mu}$-globally hypoelliptic for $\mu\geq 1$ if and only if $p+q$ does not vanish on $\Z\times \R$.
\end{corollary}
\begin{proof}
 Indeed, note that by continuity the image of $p$ is a closed interval of the real line. If $p+q$ does not vanish on $\Z\times \R$, the set $-q(\Z)$ must be disjoint from $p(\R)$ and therefore without loss of generality we may assume that $p(\R)=[\lambda,\infty)$, for some $\lambda\in\R$.  Then
 \begin{equation*}
     |p(\xi)+q(k)|\geq |p(\xi)|-|\lambda|\sim |\xi|^N,
 \end{equation*}
 for every $k\in\Z$ and sufficiently large $|\xi|$. The result then follows from Theorem \ref{theo_gh_cte_N>1}.
\end{proof}

We also note that the results above are independent of $\sigma,\mu\geq 1$.

\section{Global hypoellipticity of first-order tube-type operators}\label{sec_gh2}

In this section, we obtain a characterization for the global regularity of operators of the form
\[P=\partial_t + (a(t)+ib(t))\partial_x + q(t),\]
where $a,b,q\in\mathscr{G}^\sigma(\T^1)$ are Gevrey functions on $\T^1$, and $a,b$ are real-valued.

Due to a loss of Gevrey regularity on the $t$-variable that arises when $a$ or $b$ are non-constant, we consider an alternative notion of regularity, as in  \cite{AvCap2022}, defined as follows. First, fix $\mu\geq 1/2$ and consider the space
\[\mathscr{F}_\mu(\G) = \bigcup_{\sigma\geq 1}\mathcal{S}_{\sigma,\mu,\sigma}(\G),\]
endowed with the inductive limit topology.

As a set, $\mathscr{F}_\mu(\G)$ coincides with the union of the spaces $\S$ over all $\sigma\geq 1$. Indeed, first note that the inclusion
\[\mathscr{F}_\mu(\G)\subset \bigcup_{\sigma\geq 1}\S\]
is clear. On the other hand, if $f$ belongs to $\bigcup_{\sigma\geq 1}\S$, then there are $\sigma_0\geq 1$ and $C>0$ such that $f\in\mathcal{S}_{\sigma_0,\mu,C}(\G)$. Hence, if we take $\sigma'=\max\{\sigma_0,C\}$, we have that $f\in\mathcal{S}_{\sigma',\mu,\sigma'}(\G)\subset \mathscr{F}_\mu(\G)$, proving that the reverse inclusion also holds.

Now, notice that if $\sigma_+>\sigma$, then the inclusion $\mathcal{S}_{\sigma,\mu,\sigma}(\G)\hookrightarrow \mathcal{S}_{\sigma_+,\mu,\sigma_+}(\G)$ is compact. Indeed, we can decompose the inclusion map above as the compositions of the inclusions:
\[\mathcal{S}_{\sigma,\mu,\sigma}(\G) \overset{i_1}{\longhookrightarrow} \mathcal{S}_{\sigma,\mu,\sigma_+}(\G) \overset{i_2}{\longhookrightarrow} \mathcal{S}_{\sigma_+,\mu,\sigma_+}(\G).\]

This composition is compact because $i_1$ is compact by Theorem \ref{comp_inc} and $i_2$ is bounded. This shows that $\mathscr{F}_\mu(\G)$ is a DFS space.

Let $\mathscr{F}'_\mu(\G)$ denote the topological dual of $\mathscr{F}_\mu(\G)$, which can be regarded by duality as the space
\[\bigcap_{\sigma\geq 1}\mathcal{S}'_{\sigma,\mu,\sigma}(\G),\]
endowed with the projective limit topology. Note that  $\mathscr{F}'_\mu(\G)$ coincides with the set given by the intersection of all $\mathcal{S}'_{\sigma,\mu}(\G)$ over $\sigma\geq 1$. Moreover, $\mathscr{F}'_\mu(\G)$ is a FS space. 

\begin{definition}\label{Fmu_GH_def}
    We say that a differential operator $P$ is $\mathscr{F}_\mu$-globally hypoelliptic if
    \[u\in\mathscr{F}'_\mu(\G),\ Pu\in\mathscr{F}_\mu(\G)\ \Rightarrow\ u\in\mathscr{F}_\mu(\G).\]
\end{definition}

Notice that if a differential operator $P$ is not $\mathscr{F}_\mu$-globally hypoelliptic, there exists $u\in \bigcap_{\sigma\geq 1} \SS$ such that $u\not\in \bigcup_{\sigma\geq 1}\S$ for any $\sigma\geq1$ satisfying $Pu=f\in\mathcal{S}_{\sigma',\mu}(\G)$, for some $\sigma'\geq 1$. In particular, $P$ is not $\mathcal{S}_{\sigma',\mu}$-globally hypoelliptic.

On the other hand, if $P$ is  $\mathcal{S}_{\sigma,\mu}$-globally hypoelliptic for every $\sigma\geq 1$, then $P$ is $\mathscr{F}_\mu$-globally hypoelliptic. Indeed, if $Pu=f\in \mathscr{F}_\mu(\G)$, we must have $f\in \S$, for some $\sigma\geq 1$ and hence $u\in \S\subset \mathscr{F}_\mu(\G)$.

The next proposition shows that for constant coefficient differential operators with no mixed terms on $\T^1\times \R$, both notions of global hypoellipticity presented so far coincide.

\begin{prop}\label{prop_equivalence_cte_coef}    
Let $P$ be a differential operator on $\T^1\times\R$ with constant coefficients and no mixed terms. Then $P$ is  $\mathscr{F}_\mu$-globally hypoelliptic if and only if $P$ is $\mathcal{S}_{\sigma,\mu}$-globally hypoelliptic, for any $\sigma\geq 1$. 
\end{prop}
\begin{proof}
By Remark \ref{not_sigma}, if $P$ is $\mathcal{S}_{\sigma,\mu}$-globally hypoelliptic for some $\sigma\geq 1$, then $P$ is $\mathcal{S}_{\sigma,\mu}$-globally hypoelliptic for every $\sigma\geq 1$, and consequently $P$ is $\mathscr{F}_\mu$-globally hypoelliptic by the previous discussion. Conversely, suppose that $P$ is not $\mathcal{S}_{\sigma,\mu}$-globally hypoelliptic. Then by the proof of Theorem \ref{theo_gh_necessary} we have that there exists $u\in \mathscr{F}'_\mu(\G)$ such that $Pu=0\in\mathscr{F}_\mu(\G)$, and consequently $P$ is not $\mathscr{F}_\mu$-globally hypoelliptic.
\end{proof}

\begin{corollary}\label{coro-tube-cte}
    The operator
    \[P=\partial_t + (a_0+ib_0)\partial_x + q_0,\]
    where $q_0\in\mathbb{C}$, $a_0,b_0\in\R$ is $\mathscr{F}_\mu$-globally hypoelliptic if and only if
    \begin{enumerate}
        \item[(i)] $b_0\neq 0$ and $\frac{a_0}{b_0}\Real(q_0)+\Imag(q_0)\notin\mathbb{Z}$;
        \item[(ii)] $b_0=0$ and either $\Real(q_0)\neq 0$, or both $a_0=\Real(q_0)=0$ and $\operatorname{Im}(q_0)\notin\Z$.
    \end{enumerate}
\end{corollary}

\subsection{Real case} First, let us consider the case $b\equiv 0$, that is, operators of the form
\[P=\partial_t+a(t)\partial_x+q(t),\]
where $a,q\in\mathscr{G}^{\sigmac}(\T^1)$, for some fixed $\sigmac\geq 1$, and $a$ is real-valued. Consider the constants
\[a_0=\dfrac{1}{2\pi}\int_{0}^{2\pi}a(t)\,\mathrm{d}t\in\R\quad\text{and}\quad q_0=\dfrac{1}{2\pi}\int_0^{2\pi}q(t)\,\mathrm{d}t\in\mathbb{C}.\]

We also set
\[A(t)=\int_0^t a(s)\,\mathrm{d}s - a_0t\quad\text{and}\quad Q(t)= \int_0^t q(s)\,\mathrm{d}s - q_0t. \]

\begin{lemma}\label{lemma_ca}
There exist $C_{0,a},C_{1,a}\geq 1$ such that for every $L>0$, $\beta,\gamma\in\N_0$ we have that
    \begin{equation}\label{est_ca}
        e^{-L|\xi|^{\frac{1}{\mu}}}\left|\partial_\xi^\beta\partial_t^\gamma e^{i\xi A(t)}\right|\leq   C_{0,a} C_{1,a}^{\beta+\gamma} \gamma!^{\sigmac+\mu},
    \end{equation}
    for every $(t,\xi)\in\T^1\times\R$.
\end{lemma}
\begin{proof}
    It is easy to see that $A\in\mathscr{G}^{\sigmac}(\T^1)$, so there exist $C_0,C\geq 1$ such that
    \begin{equation}\label{ineq_gevrey_A}
        |\partial_t^\ell A(t)|\leq C_0C^{\ell}\ell!^{\sigmac},
    \end{equation}
    for every $\ell\in\N_0$,  $t\in\T^1$. Note that if $\gamma=0$, then 
    \begin{align*}
         e^{-L|\xi|^{\frac{1}{\mu}}}\left|\partial_\xi^\beta e^{i\xi A(t)}\right|&=         e^{-L|\xi|^{\frac{1}{\mu}}}\left|A(t)\right|^\beta \leq C_0^\beta C^\beta,
    \end{align*}
    for every $(t,\xi)\in\T^1\times\R$. 

 Next consider $\gamma\geq 1$. Then, by Fa\`a di Bruno's Formula (Lemma \ref{faa}) and inequality \eqref{ineq_gevrey_A} we have that
    \begin{align*}
        & e^{-L|\xi|^{\frac{1}{\mu}}}|\partial_\xi^{\beta}\partial_t^\gamma e^{i\xi A(t)}|\\&\leq e^{-L|\xi|^{\frac{1}{\mu}}}\left|\sum_{\Delta(\gamma)}\frac{\gamma!}{\tau!}\partial_\xi^\beta\left[e^{i\xi A(t)}\prod_{\ell=1}^{\gamma}\left[\frac{\partial_t^\ell(i\xi A(t))}{\ell!}\right]^{\tau_\ell}\right]\right|\\
        & =e^{-L|\xi|^{\frac{1}{\mu}}}\left|\sum_{\Delta(\gamma)}\frac{\gamma!}{\tau!}\sum_{\ell'=0}^{\min\{|\tau|,\beta\}}\binom{\beta}{\ell'}e^{i\xi A(t)}(iA(t))^{\beta-\ell'}\frac{|\tau|!}{(|\tau|-\ell')!}\xi^{|\tau|-\ell'}\prod_{\ell=1}^{\gamma}\left[\frac{i\partial_t^\ell A(t)}{\ell!}\right]^{\tau_\ell}\right|\\
        &\leq \sum_{\Delta(\gamma)}\frac{\gamma!}{\tau!}\sum_{\ell'=0}^{\min\{|\tau|,\beta\}}\binom{\beta}{\ell'}(1+|A(t)|)^{\beta}\left(\frac{\mu}{L}\right)^{\mu(|\tau|-\ell')}(|\tau|-\ell')!^{\mu-1}|\tau|!\prod_{\ell=1}^{\gamma}\left[\frac{|\partial_t^\ell A(t)|}{\ell!}\right]^{\tau_\ell}\\
        &\leq 2^{\beta}\left(1+\left(\frac{\mu}{L}\right)^{\mu}\right)^{\gamma}\sup_{t\in\T^1}(1+|A(t)|)^\beta\gamma!^{\mu}\sum_{\Delta(\gamma)}\frac{\gamma!}{\tau!}|\tau|!\prod_{\ell=1}^{\gamma}\left[\frac{C_0C^{\ell}\ell!^{\sigmac}}{\ell!}\right]^{\tau_\ell}\\
        &\leq  \tilde{C}_a^{\beta+\gamma}C^{\gamma}\gamma!^{\mu}\sum_{\Delta(\gamma)}\frac{\gamma!}{\tau!}C_0^{|\tau|} |\tau|!^{\sigmac}\prod_{\ell=1}^{\gamma}\ell!^{({\sigmac}-1)\tau_\ell},
    \end{align*}
    with $\Delta(\gamma)$ as defined in \eqref{Delta_set_defi} and $\tilde{C}_a=2\displaystyle\sup_{t\in\T^1}(1+|A(t)|)\left(1+\left(\frac{\mu}{L}\right)^{\mu}\right)$. Here, we used Lemma \ref{lemma-exponential} on the fourth line. Then, applying Lemmas \ref{lemma_algebraic_1} and \ref{lemma_algebraic_2} we have that
    \begin{align*}
         e^{-L|\xi|^{\frac{1}{\mu}}}|\partial_t^\gamma e^{i\xi A(t)}|&\leq \tilde{C}_a^{\beta+\gamma}C^{\gamma}\gamma!^{\mu}\sum_{\Delta(\gamma)}\frac{\gamma!}{\tau!}C_0^{|\tau|}|\tau|!\gamma!^{({\sigmac}-1)}\\
        &\leq \tilde{C}_a^{\beta+\gamma}C^{\gamma}\gamma!^{{\sigmac}+\mu}\sum_{\Delta(\gamma)}\frac{|\tau|!}{\tau!}C_0^{|\tau|}\\
        &\leq \tilde{C}_a^{\beta+\gamma}C^{\gamma}\gamma!^{{\sigmac}+\mu} C_0(1+C_0)^{\gamma-1}\\
        & \leq C_0(\tilde{C}_aC(1+C_0))^{\beta+\gamma}\gamma!^{{\sigmac}+\gamma}\\
        & \leq  C_{0,a}C_{1,a}^{\beta+\gamma}\gamma!^{{\sigmac}+\mu},
    \end{align*}
    for every $(t,\xi)\in\T^1\times\R$, where $C_{0,a}=C_0$ and $C_{1,a}=\tilde{C}_aC(1+C_0)$. Together with the case $\gamma=0$, this implies in inequality \eqref{est_ca}, as claimed.
\end{proof}
\begin{prop}
    The operator $\Psi_a:\mathscr{F}_\mu(\g)\to \mathscr{F}_\mu(\g)$ given by
    \[ (\Psi_a u)(t,x) = \frac{1}{2\pi}\int_{\R} e^{i\xi A(t)}\widehat{u}(t,\xi)e^{i\xi x}\,\mathrm{d}\xi  \]
    is well-defined and is and automorphism of both $\Smu$ and $\SSmu$. Moreover
    \[P_0\circ\Psi_a = \Psi_a\circ P,\]
    where
    \begin{equation}\label{P0} 
    P_0 = \partial_t + a_0\partial_x + q(t).
    \end{equation}
\end{prop}
\begin{proof}
    First we will prove that $\Psi_a$ is well-defined. Indeed, first notice that given $u\in\SSmu$, we have that $\widehat{u}\in\SSmu$. Therefore for every $\tilde{C},\tilde L>0$ and $\sigma\geq 1$, there exists $B=B(\tilde{C},L,\sigma)>0$ such that 
    \begin{equation}\label{conjugation_a_distrib}
        |\langle \widehat{u},\varphi\rangle| \leq B\sup_{\beta,\gamma\in\N_0}\sup_{(t,x)\in\g} \tilde C^{-\beta-\gamma}\gamma!^{-(\sigma+\mu)}\beta!^{-\mu}e^{\tilde L|\xi|^{\frac{1}{\mu}}}|\partial_\xi^\beta\partial_t^\gamma\varphi(t,\xi)|,
    \end{equation} 
    for all $\varphi\in\mathcal{S}_{\sigma,\mu}(\g)$. Second, notice that given $\varphi_j\to0$ in $\mathcal{S}_{\sigma',\mu}(\g)$ for some $\sigma'\geq 1$, there exist a sequence $C_j\to 0$ and constants $C\geq 1$, $L>0$ such that
    \begin{equation*}
       \sup_{\beta,\gamma\in\N_0} C^{-\beta-\gamma}\gamma!^{-\sigma'}\beta!^{-\mu}\sup_{(t,x)\in\T^1\times\R}e^{2L|x|^{\frac{1}{\mu}}}|\partial_x^\beta\partial_t^\gamma \varphi_j(t,x)|\leq C_j,
    \end{equation*}
    which implies that
        \begin{equation*}
       \sup_{(t,x)\in \T^1\times\R}|e^{2L|x|^{\frac{1}{\mu}}}\partial_x^\beta\partial_t^\gamma \varphi_j(t,x)|\leq C_j C^{\beta+\gamma}\gamma!^{\sigma'}\beta!^{\mu},
    \end{equation*}
    for every $\alpha,\beta,\gamma\in\N_0$. Taking $\sigma=\max\{\sigmac,\sigma'\},\ \tilde L=L$ and $\tilde C=2C_{1,a}C$ in \eqref{conjugation_a_distrib},  by Lemma \ref{lemma_ca} we have that
    \begin{align*}
        &|\langle e^{i\xi A(t)}\widehat{u}(t,\xi),\varphi_j(t,\xi)\rangle|\\
        & =|\langle \widehat{u}(t,\xi),e^{i\xi A(t)}\varphi_j(t,\xi)
        \rangle|\\
        & \leq B\sup_{\beta,\gamma\in\N_0} \tilde{C}^{-\beta-\gamma}\gamma!^{-(\sigma+\mu)}\beta!^{-\mu}\sup_{(t,x)\in\T^1\times\R}e^{L|\xi|^{\frac{1}{\mu}}}|\partial_\xi^{\beta}\partial_t^{\gamma}[e^{i\xi A(t)}\varphi_j(t,\xi)]|\\
        &\leq B\sup_{\beta,\gamma\in\N_0} \tilde{C}^{-\beta-\gamma}\gamma!^{-(\sigma+\mu)}\beta!^{-\mu} \sum_{\ell=0}^\gamma\sum_{\ell'=0}^\beta \binom{\gamma}{\ell}\binom{\beta}{\ell'}\sup_{(t,x)\in\T^1\times\R}e^{-L|\xi|^{\frac{1}{\mu}}}|\partial_\xi^{\ell'}\partial_t^{\ell}e^{i\xi A(t)}|\\
        &\phantom{\leq} \times\sup_{(t,x)\in\T^1\times\R}e^{2L|\xi|^{\frac{1}{\mu}}}|\partial_\xi^{\beta-\ell'}\partial_t^{\gamma-\ell}\varphi_j(t,\xi)|\\
        &\leq B\sup_{\beta,\gamma\in\N_0} \tilde{C}^{-\beta-\gamma}\gamma!^{-(\sigma+\mu)}\beta!^{-\mu}\\
         &\phantom{\leq} \times\sum_{\ell=0}^\gamma \sum_{\ell'=0}^\beta\binom{\gamma}{\ell} \binom{\beta}{\ell'}C_{0,a}C_{1,a}^{\ell+\ell'}\ell!^{(\sigmac+\mu)}C_j C^{(\beta-\ell')+(\gamma-\ell)}(\gamma-\ell)!^{\sigma'}(\beta-\ell')!^{\mu}\\
         &\leq B\sup_{\beta,\gamma\in\N_0} \tilde{C}^{-\beta-\gamma}\gamma!^{-(\sigma+\mu)}\beta!^{-\mu}\sum_{\ell=0}^\gamma \sum_{\ell'=0}^\beta\binom{\gamma}{\ell} \binom{\beta}{\ell'}C_{0,a}C_{1,a}^{\beta+\gamma}\gamma!^{(\sigma+\mu)}C_j C^{\beta+\gamma}\beta!^{\mu}\\
         &\leq C_jBC_{0,a} C_{1,a}\sup_{\beta,\gamma\in\N_0}{\left(\frac{2C_{1,a}C}{\tilde{C}}\right)}^{\beta+\gamma}\\
         &= C_jBC_{0,a} C_{1,a}.
    \end{align*}
    
     Therefore we conclude that
     \[\lim_{j\to\infty}\langle e^{i\xi A(t)}\widehat{u}(t,\xi),\varphi_j(t,\xi)\rangle=0,\]
     which gives us that $e^{i\xi A(t)}\widehat{u}(t,\xi)\in\mathcal{S}'_{\sigma',\mu}(\T^1\times\R)$. Since $\sigma'\geq 1$ was arbitrary, we conclude that $e^{i\xi A(t)}\widehat{u}(t,\xi)\in \SSmu$. Therefore $\Psi_au\in\SSmu$. Clearly $\Psi_a^{-1}=\Psi_{-a}$, so $\Psi_a$ is an automorphism for $\SSmu$. Next we will prove that $\Psi_a$ maps $\Smu$ bijectively into $\Smu$. Indeed, let $u\in\Smu$. Then since $\widehat{u}\in\Smu$, there exist $C_0,C_1,L>0$ such that 
     \begin{equation*}
         e^{2L|\xi|^{\frac{1}{\mu}}}|\partial_\xi^{\beta}\partial_t^{\gamma}\widehat{u}(t,\xi)|\leq C_0 C_1^{\beta+\gamma}\beta!^{\mu}\gamma!^{\sigma},
     \end{equation*}
     for every $(t,\xi)\in\T^1\times\R$, $\alpha,\beta,\gamma\in\N_0$, and some $\sigma\geq 1$. Let $\sigma'=\max\{\sigmac,\sigma\}$. Then
     \begin{align*}
         e^{L|\xi|^{\frac{1}{\mu}}}|\partial_\xi^{\beta}\partial_t^\gamma[e^{i\xi A(t)}\widehat{u}(t,\xi)]|&\leq  \sum_{\ell=0}^\gamma\sum_{\ell'=0}^{\beta}\binom{\gamma}{\ell}\binom{\beta}{\ell'}e^{-L|\xi|^{\frac{1}{\mu}}}|\partial_\xi^{\ell'}\partial_t^\ell e^{i\xi A(t)}|e^{2L|\xi|^{\frac{1}{\mu}}}|\partial_\xi^{\beta-\ell'}\partial_t^{\gamma-\ell}\widehat{u}(t,\xi)|\\
         &\leq \sum_{\ell=0}^\gamma\sum_{\ell'=0}^{\beta}\binom{\gamma}{\ell}\binom{\beta}{\ell'} C_{0,a}C_{1,a}^{\ell'+\ell}\ell!^{\sigmac+\mu}C_0 C_1^{\beta-\ell'+\gamma-\ell}(\beta-\ell')!^{\mu}(\gamma-\ell)!^{\sigma+\mu}\\
         &\leq C_{0,a}C_0\sum_{\ell=0}^\gamma\sum_{\ell'=0}^{\beta}\binom{\gamma}{\ell}\binom{\beta}{\ell'} C_{1,a}^{\beta+\gamma}C_1^{\beta+\gamma}\beta!^{\mu}\gamma!^{\sigma'+\mu}\\
         &\leq C_{0,a}C_02^{\beta+\gamma}C_{1,a}^{\beta+\gamma}C_1^{\beta+\gamma}\beta!^{\mu}\gamma!^{\sigma'+\mu}\\
        &=\tilde{C}_0'\tilde{C}_1^{\beta+\gamma}\beta!^{\mu} \gamma!^{\sigma'+\mu},
         \end{align*}
        for every $(t,\xi)\in\T^1\times\R$, $\alpha,\beta,\gamma\in\N_0$, where $\tilde{C}_0=C_{0,a}C_0$ and $\tilde{C}_1=2C_{1,a}C_1$ and we assumed without loss of generality that $C_1\geq 1$. Therefore $e^{i\xi A(t)}\widehat{u}(t,\xi)\in \mathcal{S}_{\sigma'+\mu,\mu}(\g)\subset\Smu$ and consequently $\Psi_au\in\Smu$. Also, since $\Psi_a^{-1}=\Psi_{-a}$, we see that $\Psi_a$ acts bijectively over $\Smu$ and therefore is an automorphism. 

          Finally, notice that 
     \begin{align*}
         & \widehat{L_{0}\circ\Psi_a u}(t,\xi) \\
         & =e^{i\xi A(t)}i\xi(a(t)-a_0)\widehat{u}(t,\xi)+e^{i\xi A(t)}\partial_t\widehat{u}(t,\xi)+e^{i\xi A(t)}a_0(i\xi)\widehat{u}(t,\xi)+e^{i\xi A(t)}q(t)\widehat{u}(t,\xi) \\
         & =e^{i\xi A(t)}(\partial_tu(t,\xi)+a(t)(i\xi)u(t,\xi)+q(t)u(t,\xi)) \ =\ \widehat{\Psi_a\circ L u}(t,\xi),
     \end{align*}
     for every $(t,\xi)\in\T^1\times \R$. Therefore we conclude that $L_0\circ\Psi_a=\Psi_a\circ L$, as claimed.
\end{proof}

\begin{prop}\label{prop_Psiq}
    The operator $\Psi_q:\sscoef\to\sscoef$ given by
    \[(\Psi_q u)(t,x) = e^{Q(t)}u(t,x)\]
    is well-defined and is an automorphism of $\sscoef$ and $\scoef$. Moreover,
    \[P_{00}\circ\Psi_q = \Psi_q\circ P_0,\]
    where $P_0$ is given by \eqref{P0} and
    \begin{equation}\label{P00}
    P_{00} = \partial_t+a_0\partial_x+q_0.
    \end{equation}
\end{prop}
\begin{proof}
    First, let us prove that $\Psi_q$ is well-defined. Notice that given $u\in\sscoef$, for every  $\tilde C>0$, there exists $B=B(\tilde C)>0$ such that 
    \begin{equation*}
        |\langle u,\varphi\rangle| \leq B\sup_{\alpha,\beta\in\N_0^n}\sup_{\gamma\in\N_0^m} \tilde C^{-|\alpha+\beta|-|\gamma|}\gamma!^{-\sigmac}(\alpha!\beta!)^{-\mu}\sup_{(t,x)\in\G}|x^\alpha\partial_x^\beta\partial_t^\gamma\varphi(t,x)|,
    \end{equation*} 
    for all $\varphi\in\scoef$. Also note that for all $\varphi_j\to0$ in $\scoef$, there exist a sequence $C_j\to0$ and a constant $C\geq 1$ such that
    \begin{equation*}
       \sup_{\alpha,\beta,\gamma\in\N_0}\sup_{(t,x)\in\T^1\times\R} {C}^{-\alpha-\beta-\gamma|}\gamma!^{-\sigmac}(\alpha!\beta!)^{-\mu}\left|x^\alpha\partial_x^\beta\partial_t^\gamma \varphi_j(t,x)\right|\leq C_j,
    \end{equation*}
    for all $j\in\N$, which implies that
        \begin{equation*}
       \sup_{(t,x)\in\T^1\times\R} \left|x^\alpha\partial_x^\beta\partial_t^\gamma \varphi_j(t,x) \right| \leq C_j {C}^{\alpha+\beta+\gamma}\gamma!^{\sigmac}(\alpha!\beta!)^{\mu},
    \end{equation*}
    for every $\alpha,\beta,\gamma\in\N_0$. 
    Third, note that $Q(t)\in\mathscr{G}^{\sigmac}(\T^1)$ and the Gevrey classes are also closed under composition. Since the exponential function is analytic, we have that $e^{Q(t)}\in\mathscr{G}^{\sigmac}(\T^1)$. Thus, there exist $C_{0,q},C_{1,q}\geq 1$ such that
    \[\sup_{t\in\T^1}|\partial_t^\gamma e^{Q(t)}|\leq C_{0,q}C_{1,q}^\gamma\gamma!^{\sigmac},\]
    for all $\gamma\in\N_0$.
    
    Therefore, given $\varphi_j\in\scoef$, $\varphi_j\to0$ in $\scoef$, we have that
    \begin{align*}
        |\langle e^{Q(t)}{u}(t,x),\varphi_j(t,x)\rangle|&=|\langle {u}(t,x),e^{Q(t)}\varphi_j(t,x)
        \rangle|\\
        & \leq B\sup_{\alpha,\beta\in\N_0}\sup_{\gamma\in\N_0} \tilde C^{-\alpha-\beta-\gamma}\gamma!^{-\sigmac}(\alpha!\beta!)^{-\mu}\!\!\!\sup_{(t,x)\in\T^1\times\R}|x^{\alpha}\partial_x^{\beta}\partial_t^{\gamma}[e^{Q(t)}\varphi_j(t,x)]|\\
        &\leq B\sup_{\alpha,\beta\in\N_0}\sup_{\gamma\in\N_0} \tilde C^{-\alpha-\beta-\gamma}\gamma!^{-\sigmac}(\alpha!\beta!)^{-\mu}\\
        &\phantom{\leq} \times\sum_{\ell=0}^\gamma \binom{\gamma}{\ell}\sup_{t\in\T^1}|\partial_t^{\ell }e^{Q(t)}|\sup_{(t,x)\in\T^1\times\R}|x^{\alpha}\partial_x^{\beta}\partial_t^{\gamma-\ell}\varphi_j(t,x)|\\
        &\leq B\sup_{\alpha,\beta\in\N_0}\sup_{\gamma\in\N_0} \tilde C^{-\alpha-\beta-\gamma}\gamma!^{-\sigmac}(\alpha!\beta!)^{-\mu}\\
         &\phantom{\leq} \times\sum_{\ell=0}^\gamma \binom{\gamma}{\ell} C_{0,q}C_{1,q}^\ell \ell!^{\sigmac}C_j {C}^{\alpha+\beta+(\gamma-\ell)}(\gamma-\ell)!^{\sigmac}(\alpha!\beta!)^{\mu}\\
         &\leq C_jC_{0,q}B\sup_{\alpha,\beta\in\N_0}\sup_{\gamma\in\N_0}\left(\frac{2C_{0,q}C}{\tilde C}\right)^{\alpha+\beta+\gamma}.
    \end{align*}
    
    Taking $\tilde C=2C_{0,q}C$, since $C_j\to 0$, we have that
    \[\lim_{j\to\infty}\langle e^{Q(t)}u(t,x),\varphi_j(t,x)\rangle=0,\]
    which implies $e^{Q(t)}u(t,x)\in\sscoef$. Clearly $\Psi_q^{-1}=\Psi_{-q}$, therefore $\Psi_q$ is an automorphism for $\sscoef$. 
    
    The fact that $\Psi_q$ maps $\scoef$ to $\scoef$ follows from Proposition \ref{prop_closed_under_product}. Again since $\Psi_q^{-1}=\Psi_{-q}$, the it is clear that $\Psi_qu^{-1}\in\scoef$ and so $\Psi_q$ is an automorphism. Finally, we have that 
     \begin{align*}
         P_{00}\circ\Psi_q u(t,x)&=e^{Q(t)}(q(t)-q_0)u(t,x)+e^{Q(t)}\partial_tu(t,x)+e^{Q(t)}a_0\partial_xu(t,x)+e^{Q(t)}q_0u(t,x)\\
         &=e^{Q(t)}(\partial_tu(t,x)+a_0\partial_xu(t,x)+q(t)u(t,x))\\
         &=\Psi_q\circ P_0 u(t,x),
     \end{align*}
     as claimed.
\end{proof}

\begin{obs}
    Notice that the proof of the previous result implies that the multiplication by a function belonging to $\mathscr{G}^{\sigma}(\T^1)$ is a continuous map from $\mathcal{S}_{\sigma,\mu}(\g)$ to $\mathcal{S}_{\sigma,\mu}(\g)$, for every $\sigma\geq 1$. 
\end{obs}
\begin{corollary}
    The operator $\Psi_q$ in Proposition \ref{prop_Psiq} is an automorphism of both $\SSmu$ and $\Smu$.
\end{corollary}
\begin{proof}
    Let $u\in \SSmu= \bigcap_{\sigma\geq 1}\mathcal{S}_{\sigma,\mu}'(\g)$. By the inclusions $\mathcal{S}_{\sigma_1,\mu}'(\g)\subset\mathcal{S}_{\sigma_2,\mu}'(\g)$ for $\sigma_1\geq \sigma_2$, it is enough to prove that $u\in\bigcap_{\sigma\geq \sigmac}\mathcal{S}_{\sigma,\mu}'(\g)$. Indeed, since $q\in \mathcal{S}_{\sigma,\mu}(\g)$ for every $\sigma\geq\sigmac$, by Proposition \ref{prop_Psiq} we have that $\Psi_q u\in\mathcal{S}_{\sigmac,\mu}'(\g)$ for every $\sigma\geq\sigmac$ and therefore $\Psi_qu\in\SSmu$. Since $\Psi_{q}^{-1}=\Psi_{-q}$, we conclude that $\Psi_q$ is an automorphism of $\SSmu$.
    
    Now consider $u\in\Smu$. Then $u\in\mathcal{S}_{\sigma}(\g)$, for some $\sigma\geq1$. If $\sigma\geq \sigmac$, then $q\in \mathcal{S}_{\sigma}(\g)$ as well and therefore by Proposition \ref{prop_Psiq} we have that $\Psi_qu\in\mathcal{S}_{\sigma}(\g)\subset\Smu$. On the other hand, if $\sigma<\sigmac$, then $u\in \scoef$ and $\Psi_qu\in \scoef\subset\Smu$, proving that $\Psi_q(\Smu)\subset\Smu$. Once again we conclude that $\Psi_q$ is an automorphism of $\Smu$ by the fact that $\Psi_{q}^{-1}=\Psi_{-q}$.
\end{proof}

\begin{prop}\label{gh_P_P00}
    The operator $\Psi=\Psi_a\circ\Psi_q$ is an automorphism of $\Smu$ and $\SSmu$ and satisfies $P_{00}\circ\Psi=\Psi\circ P$. Moreover, $P$ is $\mathscr{F}_{\mu}$-globally hypoelliptic if and only if $P_{00}$ is $\mathscr{F}_{\mu}$-globally hypoelliptic.
\end{prop}
\begin{proof}
    The claims that $\Psi$ is an automorphism of $\Smu$ and $\SSmu$ and $P_{00}\circ \Psi=\Psi\circ P$ are direct consequences of the last two Propositions.
    
    Now suppose that $P_{00}$ is $\mathscr{F}_\mu$-globally hypoelliptic and let $u\in\SSmu$ be such that $Pu=f\in\Smu$. Then $\Psi\circ Pu=P_{00}\circ \Psi u=\Psi f\in\Smu$, since $\Psi$ is an automorphism of $\Smu$. Since $P_{00}$ is $\mathscr{F}_{\mu}$-globally hypoelliptic, this implies that $\Psi u\in\Smu$. Consequently, since $\Psi$ is an automorphism of $\Smu$, we have that $u\in\Smu$, which proves that $P$ is $\mathscr{F}_{\mu}$-globally hypoelliptic. The converse follows similarly, by considering that $P\circ \Psi^{-1}=\Psi^{-1}\circ P_{00}$.
\end{proof}

Finally, Proposition \ref{gh_P_P00} and Corollary \ref{coro-tube-cte} gives us the following:

\begin{theorem}
    The operator $P$ is $\mathscr{F}_{\mu}$-globally hypoelliptic if and only if one of the following conditions hold:
    \begin{enumerate}
        \item[(i)] $\Real(q_0)\neq 0$;
        \item[(ii)] $a_0=\Real(q_0)=0$ and $\operatorname{Im}(q_0)\notin\Z$.
    \end{enumerate}
\end{theorem}

\subsection{Complex case} Finally, we study the global hypoellipticity for the general case of operators of the form
\[P=\partial_t+(a(t)+ib(t))\partial_x+q(t),\]
where $a,b,q\in\mathscr{G}^{\sigmac}(\T^1)$ are Gevrey functions on $\T^1$ for some fixed $\sigmac\geq 1$, with $a$ and $b$ real-valued, and $b\not\equiv 0$. From the results on the previous subsection, it is enough to consider the case where $a\equiv a_0\in\R$ and $q\equiv q_0\in\mathbb{C}$. Consider also the operator
\[\tilde P = \partial_t + c_0\partial_x+q_0,\]
where
\[c_0 = a_0 + \dfrac{i}{2\pi}\int_{0}^{2\pi}b(t)\,\mathrm{d}t = a_0+ib_0.\]

\begin{theorem}
    For $\mu >1$, the operator $P$ is $\mathscr{F}_{\mu}$-globally hypoelliptic if and only if  $b$ does not change sign and $\dfrac{a_0}{b_0}\Real(q_0)+\Imag(q_0)\not\in\Z$.
\end{theorem}

\begin{obs}
    Notice that the assumption that $b$ does not change sign implies $b_0\neq 0$. 
\end{obs}

The proof of this result will be divided in the subsequent propositions.

\begin{prop}\label{prop-L-GH-implies-tildeL}
    Let $\mu\geq 1/2$. If $P$ is $\mathscr{F}_{\mu}$-globally hypoelliptic-, then $\tilde P$ is $\mathscr{F}_{\mu}$-globally hypoelliptic.
\end{prop}
\begin{proof}
    Suppose that $\tilde{P}$ is not $\mathscr{F}_{\mu}$-globally hypoelliptic. Then by Corollary \ref{coro-tube-cte} there exists $(k_0,\xi_0)\in\Z\times\R$ such that $k+c_0\xi_0-iq_0=0$, or equivalently $c_0\xi_0-iq_0\in\Z$. 
    
    Consider the function $v$ given by
    \[v(t,x) = \exp\left\{-\int_0^{t}i\xi_0c(s)+q_0{\,\mathrm{d}s}\right\}e^{i\xi_0x}.\]

     Notice that $v$ is well-defined by the hypothesis on $(k_0,\xi_0)$. Also, $v\in\bigcap_{\sigma\geq 1}\mathcal{S}_{\sigma,\mu}(\g)'=\mathscr{F}_\mu'(\g)$ by Proposition \ref{proppartialdecayrndistribinv}.
    Moreover, $v\not\in \mathcal{S}_{\sigma,\mu}(\g)$ for any $\sigma\geq 1$, because
    \begin{equation*}
        |v(2\pi,x)|=|\exp(-2\pi i(\xi_0c_0-iq_0)|=1,
    \end{equation*}
    for every $x\in\R$. Hence, $v\in\SSmu\setminus\Smu$. Since $v$ satisfies
	$$Pv=0\in\Smu$$
	we conclude that that $P$ is not $\mathscr{F}_{\mu}$-globally hypoelliptic.
\end{proof}

Notice that by the proof above and Proposition \ref{prop_equivalence_cte_coef}  we conclude that if $\tilde{P}$ is not $\mathcal{S}_{\sigma,\mu}$-globally hypoelliptic then $P$ is not $\mathcal{S}_{\sigma,\mu}$-globally hypoelliptic, for any $\sigma\geq 1$.

\begin{prop}
     If $P$ is $\mathscr{F}_{\mu}$-globally hypoelliptic and $\mu>1$, then $b$ does not change sign.
\end{prop}
\begin{proof}
    Suppose that $b$ changes sign. If $\tilde{P}$ is not $\mathscr{F}_{\mu}$-globally hypoelliptic, then $P$ is  not $\mathscr{F}_{\mu}$-globally hypoelliptic by Proposition \ref{prop-L-GH-implies-tildeL}, hence we may assume that $\tilde{P}$ is $\mathscr{F}_{\mu}$-globally hypoelliptic.
    
    First assume that $b_0\geq 0$ and let
    \begin{equation*}
    G(t,s) \defeq\int_t^{t+s}b(w)\,\mathrm{d}w,
    \end{equation*}
    \begin{equation*}
        B\defeq \min_{s,t\in[0,2\pi]}G(t,s)=\int_{t_0}^{t_0+s_0}b(w)\,\mathrm{d}w.
    \end{equation*}
    
    Notice that $B<0$ since $b$ is continuous and changes sign. We may assume, without loss of generality, that $t_0,s_0,t_0+s_0\in(0,2\pi)$. Consider $\sigma_1>1$ and choose $\varphi\in \mathscr{G}^{\sigma_1}(\T^1)$ such that $\supp(\varphi)\subset[t_0+s_0-\delta,t_0+s_0+\delta]$, where $\delta>0$ is such that $(t_0+s_0-\delta ,t_0+s_0+\delta)\subset(0,2\pi)$, and $0\leq \varphi\leq 1$ with $\varphi(t)\equiv1$ in $[t_0+s_0-\delta/2,t_0+s_0+\delta/2]$. 
    Also, since $\mu>1$, we can take $\psi\in \mathscr{G}^\mu_c(\R)$ such that $\supp(\psi)\subset[0,\infty)$ and $\psi(\xi)\equiv1 $ on $[1,\infty)$. For each $\xi\in\R$, define $\widehat{f}(\cdot,\xi)$ to be the $2\pi$-periodic smooth extension of
    \[t\mapsto (e^{2\pi i(\xi c_0-iq_0)}-1)e^{B\xi}\varphi(t)e^{-i\xi a_0(t-t_0)}e^{-q_0(t-t_0)}\psi(\xi).\]
    
    Notice that
\begin{align}
|\partial_\xi^\beta\partial_t^\gamma\widehat{f}(t,\xi)|\leq & \sum_{\ell=0}^\gamma\sum_{\ell'=0}^\beta  \binom{\gamma}{\ell}\binom{\beta}{\ell'}|\partial_t^{\gamma-\ell}\varphi(t)||\partial_\xi^{\beta-\ell'}\psi(\xi)||\partial_\xi^{\ell'}\partial_t^{\ell}e^{\xi(B-ia_0(t-t_0)+2\pi i c_0)-q_0(t-t_0)+2\pi q_0}|\notag\\
    & +\sum_{\ell=0}^\gamma\sum_{\ell'=0}^\beta \binom{\gamma}{\ell}\binom{\beta}{\ell'} |\partial_t^{\gamma-\ell}\varphi(t)||\partial_\xi^{\beta-\ell'}\psi(\xi)||\partial_\xi^{\ell'}\partial_t^{\ell}e^{\xi(B-ia_0(t-t_0))-q_0(t-t_0)}|\label{ineq_fourer_changes_sign},
\end{align}
    for every $(t,\xi)\in\T^1\times\R$, $\beta,\gamma\in\N_0$.
    
    Next, consider the function $e^{p(t,\xi)}$, where $p(t,\xi)=ia_1\xi t+c_2\xi + c_3t+c_4$ is a complex polynomial on two variables and $a_1\in\R$. Then for $L>0$, by applying Lemma \ref{lemma-exponential} we have that
        \begin{align}
          e^{L|\xi|^{\frac{1}{\mu}}}|\partial_\xi^{\beta}\partial_t^\gamma e^{p(t,\xi)}|
          &\leq  e^{L|\xi|^{\frac{1}{\mu}}}\sum_{\ell=0}^{\min\{\gamma,\beta\}}\binom{\beta}{\ell}|\partial_{\xi}^{\beta-\ell}e^{p(t,\xi)}||\partial_\xi^\ell[(ia_1\xi +c_3)^\gamma]|\notag\\
         &= e^{L|\xi|^{\frac{1}{\mu}}}\sum_{\ell=0}^{\min\{\gamma,\beta\}}\binom{\beta}{\ell}|e^{p(t,\xi)}||ia_1t+c_2|^{\beta-\ell}\frac{\gamma!}{(\gamma-\ell)!}|ia_1\xi +c_3|^{\gamma-\ell}|a_1|^\ell\notag\\
         &\leq \sum_{\ell=0}^{\min\{\gamma,\beta\}}\binom{\beta}{\ell} e^{2\pi|\Real(c_3)|+\Real(c_4)}e^{|\Real(c_2)|\xi}C^{\beta}|a_1|^\ell2^{\gamma-\ell}\notag\\
         &\phantom{\geq} \times\left[|a_1|^{\gamma-\ell}\frac{\gamma!}{(\gamma-\ell)!}\left(\frac{\mu}{L}\right)^{\mu(\gamma-\ell)}(\gamma-\ell)!^\mu e^{2L|\xi|^{\frac{1}{\mu}}}+\frac{\gamma!}{(\gamma-\ell)!}e^{L|\xi|^{\frac{1}{\mu}}}|c_3|^{\gamma-\ell}\right]\notag\\
         &\leq (2C)^\beta e^{2\pi|\Real(c_3)|+\Real(c_4)}e^{|\Real(c_2)|\xi+2L|\xi|^{\frac{1}{\mu}}}\gamma!^{\mu}2^{\gamma}(1+|a_1|)^\gamma\left(2+\left(\frac{\mu}{L}\right)^{\mu}+|c_3|\right)^\gamma \notag\\
         &\leq  (2C)^\beta e^{2\pi|\Real(c_3)|+\Real(c_4)}e^{|\Real(c_2)|\xi+2L|\xi|^{\frac{1}{\mu}}}\gamma!^{\mu}2^{\gamma}(1+|a_1|)^\gamma\left(2+\left(\frac{\mu}{L}\right)^{\mu}+|c_3|\right)^\gamma\notag\\
         & \leq \tilde{C}_0\tilde{C}_1^{\beta+\gamma}\gamma!^{\mu}e^{|\Real(c_2)|\xi+2L|\xi|^{\frac{1}{\mu}}},\label{ineq_exp_poly}
    \end{align}
    for every $(t,\xi)\in [0,2\pi]\times\R$ and $\beta,\gamma\in\N_0$, where $C=\max_{t\in[0,2\pi]}|ia_1t+c_2|+1$, $\tilde{C}_0=e^{2\pi|\Real(c_3)|+\Real(c_4)}$, $\tilde{C}_1=\max\left\{2C,2(1+|a_1|)(2+(\frac{\mu}{L})^\mu+|c_3|)\right\}$ and we used the fact that $\mu> 1$. Evidently, both exponentials in the inequality \eqref{ineq_fourer_changes_sign} are of the form $e^{p(t,\xi)}$. Hence, from the inequalities above, we conclude that 
\begin{align*}
     &e^{L|\xi|^{\frac{1}{\mu}}}|\partial_\xi^\beta\partial_t^\gamma\widehat{f}(t,\xi)|\\
     &\leq\sum_{\ell=0}^\gamma\sum_{\ell'=0}^\beta  \binom{\gamma}{\ell}\binom{\beta}{\ell'}C_{0,\varphi}C_{1,\varphi}^{\gamma-\ell}(\gamma-\ell)!^{\sigma_1}C_{0,\psi}C_{1,\psi}^{\beta-\ell'}(\beta-\ell')!^{\mu}\tilde C_{0,1}\tilde C_{1,1}^{\ell'+\ell}\ell!^{\mu}e^{(B-2\pi b_0)\xi+L'|\xi|^{\frac{1}{\mu}}}\\
     &\phantom{\geq} +\sum_{\ell=0}^\gamma\sum_{\ell'=0}^\beta  \binom{\gamma}{\ell}\binom{\beta}{\ell'}C_{0,\varphi}C_{1,\varphi}^{\gamma-\ell}(\gamma-\ell)!^{\sigma_1}C_{0,\psi}C_{1,\psi}^{\beta-\ell'}(\beta-\ell')!^{\mu}\tilde C_{0,2}\tilde C_{1,2}^{\ell'+\ell}\ell!^{\mu}e^{B\xi+L'|\xi|^{\frac{1}{\mu}}},
\end{align*}
 for every $(t,\xi)\in\T^1\times\R$, $\beta,\gamma\in\N_0$, and for some constants $C_{0,\varphi}, C_{1,\varphi},C_{0,\psi},C_{1,\psi}>0$ depending on $\varphi$ and $\psi$, and $\tilde C_{0,j}, \tilde C_{1,j}\geq 1$, $j=1,2$ given by \eqref{ineq_exp_poly}. Therefore, since $\mu>1$ and $\supp \widehat{f}(\cdot,\xi)\subset [0,\infty)$, for every $(t,\xi)\in\T^1\times\R$ and $\beta,\gamma\in\N_0$ we have that
\begin{align*}
     e^{L|\xi|^{\frac{1}{\mu}}}|\partial_\xi^\beta\partial_t^\gamma\widehat{f}(t,\xi)|&\leq 2K_0C_{0,\varphi}C_{0,\psi}C_{11}2^\beta 2^\gamma C_{1,\varphi}^{\gamma}C_{1,\psi}^{\beta}\gamma!^{\max\{\sigma_1,\mu\}}\beta!^{\mu}C_{21}^{\beta}C_{31}^{\gamma}\\
     &\leq C_0C_1^{\beta+\gamma}\gamma!^{\max\{\sigma_1,\mu\}}\beta!^\mu,
\end{align*}
 where $C_0 = 2K_0C_{0,\varphi}C_{0,\psi}C_{11}$, $C_1= 2\max\{C_{1,\varphi} C_{1,\psi}\tilde C_{1,1}\tilde C_{1,2},1\} $ and $K_0=\sup_{\xi>0}e^{B\xi+2L|\xi|^{\frac{1}{\mu}}}$, where $K_0$ is finite since $B<0$ and $\mu>1$. Therefore we conclude that $\widehat{f}\in \mathcal{S}_{\max\{\sigma_1,\mu\},\mu}(\g)$ and consequently $f\in\mathcal{S}_{\max\{\sigma_1,\mu\},\mu}(\g)\subset\Smu$ by Corollary \ref{coro_gelf_transf}. 
    
    \medskip

    Now, notice that if $Pu=f$, then taking the partial Fourier transform with respect to the variable $x$ yields
    $$\partial_t\widehat{u}(t,\xi)+i(\xi(a_0+ib(t))-iq_0)\widehat{u}(t,\xi) = \widehat{f}(t,\xi),$$
    for $(t,\xi)\in\g$.    
    For each fixed $\xi\in\R$, by Lemma \ref{lemmaodesol}  these differential equations admit unique solution given by
    $$\widehat{u}(t,\xi) = \frac{1}{e^{2\pi i(\xi c_0-iq_0)}-1}\int_{0}^{2\pi}\widehat{f}(t+s)e^{\int_t^{t+s}i(\xi c(w)-iq_0)\,\mathrm{d}w}\,\mathrm{d}s,$$
    which in this case can be written as
    $$\widehat{u}(t,\xi) = e^{-i\xi a_0(t-t_0)}e^{-q_0(t-t_0)}\psi(\xi)\int_0^{2\pi}e^{\xi(B-G(t,s))}\varphi(t+s)\,\mathrm{d}s,$$
    since $\xi c_0-iq\not\in\Z$ for every $\xi\in\R$ by  Corollary \ref{coro-tube-cte} because $\tilde{P}$ is $\mathscr{F}_{\mu}$-globally hypoelliptic. 
    
    Notice that, since $\supp(\psi)\subset [0,+\infty)$, we have that
    $$|\widehat{u}(t,\xi)|\leq 2\pi e^{2\pi|\Real(q_0)|},$$
    for every $(t,\xi)\in\g$.  Proposition \ref{proppartialdecayrndistribinv} then implies that $u\in\SSmu$.

    On th other hand, note that if $\xi>1$ is sufficiently big, then
    \begin{align*}
        |\widehat{u}(t_0,\xi)|&=\left|\int_0^{2\pi}e^{\xi(B-G(t_0,s))}\varphi(t_0+s)\,\mathrm{d}s\right|\\
        &\geq \int_{s_0-\frac{\delta}{2}}^{s_0+\frac{\delta}{2}}e^{\xi(B-G(t_0,s))}\,\mathrm{d}s\ \geq\ K\frac{1}{\sqrt{\xi}},
    \end{align*}
    for some $K>0$, by Lemma \ref{lemmaintegralineq}. Therefore, by Proposition \ref{propdecaypartialxi}, we conclude that $\widehat{u}\not\in\mathcal{S}_{\sigma,\mu}(\g)$ for any $\sigma\geq 1$, hence $u\not\in\Smu$. Since  $\widehat{Pu}(t,\xi)=\widehat{f}(t,\xi)$ for every $(t,\xi)\in\T^1\times\R$ implies that $Pu=f$, we obtain that $P$ is not $\mathscr{F}_{\mu}$-globally hypoelliptic, as claimed.
    
    Suppose now that $b_0<0$. The proof is similar, the main difference is that now we define
    \begin{equation*}
    \tilde G(t,s) \defeq\int_{t-s}^{t}b(w)\,\mathrm{d}w,
    \end{equation*}
    \begin{equation*}
        \tilde B\defeq\max_{s,t\in[0,2\pi]}\int_{t-s}^tb(w)\,\mathrm{d}w=\int_{t_1-s_1}^{t_1}b(w)\,\mathrm{d}w.
    \end{equation*}
    
    Again, we may assume, without loss of generality, that $t_1,s_1,t_1-s_1\in(0,2\pi)$. Let $\tilde{\varphi}\in \mathscr{G}^{\sigma_1}(\T^1)$ be such that $\supp(\tilde{\varphi})\subset[t_1-s_1-\delta,t_1-s_1+\delta]$, where $\delta>0$ is such that $(t_1-s_1-\delta ,t_1-s_1+\delta)\subset(0,2\pi)$, and $0\leq \tilde \varphi\leq 1$ with $\tilde \varphi(t)\equiv1$ in $[t_1-s_1-\delta/2,t_1-s_1+\delta/2]$. Then, for each $\xi\in\R$, we define $\widehat{f}(\cdot,\xi)$ to be the $2\pi$-periodic smooth extension of
    $$t\mapsto\left(1-e^{-2\pi i(\xi c_0-iq_0)}\right)e^{-\tilde{B}\xi}\tilde{\varphi}(t)e^{-i\xi a_0(t-t_1)}e^{-q(t-t_1)}\psi(\xi),$$
    and the remaining part of the proof follows analogously, but we use formula \eqref{sol-} to study the solution $\widehat{u}(t,\xi)$.
\end{proof}
\begin{prop}
    Suppose that $\tilde P$ is $\mathscr{F}_{\mu}$-globally hypoelliptic, $b\not\equiv 0$ and $b$ does not change sign. Then, $P$ is $\mathscr{F}_{\mu}$-globally hypoelliptic for any $\mu\geq 1/2$.
\end{prop}
\begin{proof}
    Suppose that $Pu=f\in\Smu$. First note that due to the inclusions $\mathscr{G}^{\sigma_1}(\T^1)\subset \mathscr{G}^{\sigma_2}(\T^1)$ and $\mathcal{S}_{\sigma_1,\mu}(\g)\subset\mathcal{S}_{\sigma_2,\mu}(\g)$ for $\sigma_1\leq \sigma_2$, we can assume without loss of generality that $\sigmac=\sigma$, where  $f\in\s$. Then, by taking the partial Fourier transform with respect to the second variable, we see that $\widehat{u}(t,\xi)$ must be the only solution to $\widehat{Pu}(t,\xi) = \widehat{f}(t,\xi)$ given by any of the two equivalent formulas in Lemma \ref{lemmaodesol}. Suppose first $b(t)\geq 0$, which implies $ b_0>0$. For $\xi\geq 0$ consider the solution $\widehat{u}(t,\xi)$ given by the formula
    \begin{align*}
        \widehat{u}(t,\xi)=\frac{1}{e^{2\pi i(\xi c_0-iq_0)}-1}\int_{0}^{2\pi}\widehat{f}(t+s,\xi)e^{qs}e^{i\xi a_0s-\xi\int_{t}^{t+s}b(w)\,\mathrm{d}w}\,\mathrm{d}s,
    \end{align*}
    and for $\xi<0$, consider the equivalent formula
        \begin{align*}
        \widehat{u}(t,\xi)=\frac{1}{1-e^{-2\pi i(\xi c_0-iq_0)}}\int_{0}^{2\pi}\widehat{f}(t-s,\xi)e^{-qs}e^{-i\xi a_0s+\xi\int_{t-s}^{t}b(w)\,\mathrm{d}w}\,\mathrm{d}s.
    \end{align*}

    Set 
    \begin{equation*}
        \mathcal{H}_1(\xi)=1-e^{-2\pi i (\xi c_0-i q_0)}\quad\text{and}\quad 
         \mathcal{H}_2(t,s,\xi)=e^{-qs+\xi\left(\int_{t-s}^t b(w)\,\mathrm{d}w-ia_0s\right)},
    \end{equation*}
    for every $t,s\in [0,2\pi]$, $\xi\in\R$.
    Notice that $\mathcal{H}_1$ is analytic and therefore belongs to $\mathscr{G}^\sigma(\R)$.

    Since $\tilde P$ is $\mathscr{F}_{\mu}$-globally hypoelliptic, by Corollary \ref{coro-tube-cte} we have that $|k+c_0\xi-iq_0|\neq 0$ for every $(k,\xi)\in\Z\times \R$. From the proof of Theorem \ref{theo_gh_cte_Nleq1} we conclude that there exists $\delta>0$ such that 
    $|k+c_0\xi-iq_0|\geq \delta$, for every $(k,\xi)\in\Z\times \R$.

Note that by continuity that this implies that there exists $\varepsilon>0$ such that $|\mathcal{H}_1(\xi)|\geq\varepsilon$, for every $\xi\in\R$. Also, by Lemma \ref{derivative-reciprocal}, for $\ell\geq 1$ we have that
    \begin{align*}
        \partial_\xi^{\ell}[\mathcal{H}_1(\xi)^{-1}]&=\sum_{k=1}^\ell(-1)^k\binom{\ell+1}{k+1}\frac{1}{\mathcal{H}_1(\xi)^{k+1}}\partial_\xi^\ell[\mathcal{H}_1(\xi)^k]\\
        &=\sum_{k=1}^\ell(-1)^k\binom{\ell+1}{k+1}\frac{1}{\mathcal{H}_1(\xi)^{k+1}}\partial_\xi^\ell\left[\sum_{j=0}^k\binom{k}{j}(-1)^{k-j}e^{-2\pi ij(\xi c_0-iq_0)}\right]\\
        &=\sum_{k=1}^\ell(-1)^k\binom{\ell+1}{k+1}\frac{1}{\mathcal{H}_1(\xi)^{k+1}}\sum_{j=0}^k\binom{k}{j}(-1)^{k-j+\ell}(2\pi i j c_0)^\ell e^{-2\pi ij(\xi c_0-iq_0)},
    \end{align*}
for every $\xi\in\R$. Hence 
        \begin{align}
        \left|\partial_\xi^{\ell}[\mathcal{H}_1(\xi)^{-1}]\right| & \leq(2\pi \ell |c_0|)^\ell\sum_{k=1}^\ell\binom{\ell+1}{k+1}\varepsilon^{-(k+1)}\sum_{j=0}^k\binom{k}{j} e^{-2\pi j \Real(q_0)} \notag \\
        &\leq (2\pi |c_0|)^\ell e^\ell \ell!\sum_{k=1}^\ell\binom{\ell+1}{k+1}\varepsilon^{-(k+1)}\left(1+e^{-2\pi  \Real(q_0)}\right)^k \notag \\
        &\leq \varepsilon^{-1}(2\pi |c_0|)^\ell e^\ell \ell!\sum_{k=0}^{\ell}\binom{\ell}{k}\frac{\ell+1}{k+1}\left(\frac{1+e^{-2\pi  \Real(q_0)}}{\varepsilon}\right)^k \notag \\
        & \leq \varepsilon^{-1}(4\pi |c_0|)^\ell e^\ell \ell!(\ell+1)\left(1+\frac{1+e^{-2\pi  \Real(q_0)}}{\varepsilon}\right)^\ell, \label{H1}
    \end{align}
    for every $\xi<0$ and $\ell\in\N_0$, since $e^{2\pi b_0\xi}\leq 1$ and where we used
    \[\sum_{k=1}^{\ell}\binom{\ell+1}{k+1}\leq \sum_{k=0}^{\ell}\binom{\ell}{k}\frac{\ell+1}{k+1}\leq (\ell+1)\sum_{k=0}^{\ell}\binom{\ell}{k} = (\ell+1)2^\ell.\]
    
    Also, note that since $b(t)\geq 0$ for every $t\in\T^1$, we have that $\int_{t-s}^tb(w)\,\mathrm{d}w\geq 0$ and therefore
    \begin{equation*}
        |\mathcal{H}_2(t,s,\xi)|\leq e^{2\pi|\Real(q_0)|},
    \end{equation*}
for every $t,s\in [0,2\pi]$, $\xi<0$.
    
By the Fa\`a di Bruno's Formula, we have that
    \begin{align}\label{H2_1}
    \partial_t^{\ell_1}\mathcal{H}_2(t,s,\xi) = \mathcal{H}_2(t,s,\xi)\sum_{\tau\in\Delta(\ell_1)}\frac{\ell_1!}{\tau!}\xi^{|\tau|}\prod_{j_1=1}^{\ell_1}\left(\frac{1}{j_1!}\partial_t^{j_1-1}[b(t)-b(t-s)]\right)^{\tau_{j_1}}. 
    \end{align}
    
    Moreover, it is easy to check that 
    \begin{align}\label{H2_2}
        \partial_\xi^{\ell_2'}\mathcal{H}_2(t,s,\xi) = \mathcal{H}_2(t,s,\xi)\left(\int_{t-s}^tb(w)\,\mathrm{d}w-ia_0s\right)^{\ell_2'}.
    \end{align}
    
    Hence, combining \eqref{H2_1} and \eqref{H2_2}, we obtain by the Leibniz's rule that
    \begin{align*}
    \partial_\xi^{\ell_2'}\partial_t^{\ell_1}\mathcal{H}_2(t,s,\xi) =&\sum_{\tau\in\Delta(\ell_1)}\sum_{j_2=0}^{\min\{\ell_2',|\tau|\}}\binom{\ell_2'}{j_2}\mathcal{H}_2(t,s,\xi)\left(\int_{t-s}^tb(w)\,\mathrm{d}w-ia_0s\right)^{\ell_2'-j_2}\\&\times\frac{|\tau|!}{(|\tau|-j_2)!}\xi^{|\tau|-j_2}
    \frac{\ell_1!}{\tau!}\prod_{j_1=1}^{\ell_1}\left(\frac{1}{j_1!}\partial_t^{j_1-1}[b(t)-b(t-s)]\right)^{\tau_{j_1}}. 
\end{align*}

But note that
\begin{equation*}
    \left|\frac{1}{j_1!}\partial_t^{j_1-1}[b(t)-b(t-s)]\right|\leq \frac{1}{j_1!}2C_{0,b}C_{1,b}^{j_1-1}(j_1-1)!^\sigma\leq 2C_{0,b}C_{1,b}^{j_1-1}j_1!^{\sigma-1}, 
\end{equation*}
for some $C_{0,b},C_{1,b}\geq 1$. Hence
\begin{equation*}
   \prod_{j_1=1}^{\ell_1}\left(\frac{1}{j_1!}\partial_t^{j_1-1}[b(t)-b(t-s)]\right)^{\tau_{j_1}}\leq (2C_{0,b})^{|\tau|}C_{1,b}^{\ell_1}\prod_{j_1=1}^{\ell_1}j_1!^{(\sigma-1)\tau_{j_1}}. 
\end{equation*}

Therefore we can estimate
\begin{align}
    \left|\partial_\xi^{\ell_2'}\partial_t^{\ell_1}\mathcal{H}_2(t,s,\xi)\right| \leq &\   |\mathcal{H}_2(t,s,\xi)|\sum_{\tau\in\Delta(\ell_1)}\sum_{j_2=0}^{\min\{\ell_2',|\tau|\}}\binom{\ell_2'}{j_2}|s|^{\ell_2'-j_2}C_{b,a_0}^{\ell_2'-j_2}\frac{|\tau|!}{(|\tau|-j_2)!}\notag\\
    &\  \times |\xi|^{|\tau|-j_2}
    \frac{\ell_1!}{\tau!}(2C_{0,b})^{|\tau|}C_{1,b}^{\ell_1}\prod_{j_1=1}^{\ell_1}j_1!^{(\sigma-1)\tau_{j_1}}\notag\\
    \leq &\  e^{2\pi|\Real(q)|}\sum_{j_2=0}^{\min\{\ell_2',|\tau|\}}\binom{\ell_2'}{j_2}(2\pi)^{\ell_2'}C_{b,a_0}^{\ell_2'}|\xi|^{\ell_1}C_{1,b}^{\ell_1}\notag\\
    &\ \times\sum_{\tau\in\Delta(\ell_1)} 
    \frac{\ell_1!}{\tau!}(2C_{0,b})^{|\tau|}|\tau|!^\sigma\prod_{j_1=1}^{\ell_1}j_1!^{(\sigma-1)\tau_{j_1}}\notag\\
        \leq &\ e^{2\pi|\Real(q)|}(4\pi C_{b,a_0})^{\ell_2'}|\xi|^{\ell_1}C_{1,b}^{\ell_1}\sum_{\tau\in\Delta(\ell_1)} 
    \frac{\ell_1!}{\tau!}(2C_{0,b})^{|\tau|}|\tau|!\ell_1!^{\sigma-1}\notag\\
    \leq &\ e^{2\pi|\Real(q)|}(4\pi C_{b,a_0})^{\ell_2'}|\xi|^{\ell_1}C_{1,b}^{\ell_1}\ell_1!^\sigma 2C_{0,b}(1+2C_{0,b})^{\ell_1-1},\label{H2}
\end{align}
for every $t,s\in[0,2\pi]$ and $\xi<0$, where $C_{b,a_0} = \displaystyle\max_{w\in[0,2\pi]} |b(w)| + |a_0|$ and in the last two lines we applied Lemmas \ref{lemma_algebraic_1} and \ref{lemma_algebraic_2}.

Next note that
\begin{align}\label{eq-util-final-proof}
    \partial_\xi^\beta\partial_t^{\gamma}\widehat{u}(t,\xi)=\sum_{|\ell'|=\beta} \binom{\beta}{\ell'}\partial_\xi^{\ell'_1}[\mathcal{H}_1(\xi)^{-1}]\sum_{|\ell|=\gamma}\binom{\gamma}{\ell}\int_0^{2\pi}\partial_\xi^{\ell'_2}\partial_t^{\ell_1}\mathcal{H}_2(t,s,\xi)\partial_\xi^{\ell'_3}\partial_t^{\ell_2}\widehat{f}(t-s,\xi)\,\mathrm{d}s,
\end{align}
for every $(t,\xi)\in\T^1\times\R$. Since $\widehat{f}\in\s$, there exist $L>0$, $C_{0,f},C_{1,f}\geq 1$ such that
\begin{equation*}
    e^{L|\xi|^{\frac{1}{\mu}}}|\partial_\xi^\beta\partial_t^\gamma \widehat{f}(t,\xi)|\leq C_{0,f}C_{1,f}^{\beta+\gamma}\beta!^\mu\gamma!^\sigma,
\end{equation*}
for every $(t,\xi)\in\T^1\times\R$, $\beta,\gamma\in\N_0$.
Applying \eqref{H1}, \eqref{H2} and the inequality above to $\eqref{eq-util-final-proof}$, yields
\begin{align*}
    &e^{\frac{L}{2}|\xi|^{\frac{1}{\mu}}}|\partial_\xi^\beta\partial_t^{\gamma}\widehat{u}(t,\xi)|\\
    &\leq \sum_{|\ell'|=\beta} \binom{\beta}{\ell'}\varepsilon^{-1}(4\pi |c_0|)^{\ell_1'} e^{\ell_1'} \ell_1'!(\ell_1'+1)\left(1+\frac{1+e^{-2\pi  \Real(q_0)}}{\varepsilon}\right)^{\ell_1'}\sum_{|\ell|=\gamma}\binom{\gamma}{\ell}2\pi e^{2\pi|\Real(q)|}(4\pi C_{b,a_0})^{\ell_2'}\\
    & \phantom{\leq} \times |\xi|^{\ell_1}e^{-\frac{L}{2}|\xi|^{\frac{1}{\mu}}}C_{1,b}^{\ell_1}\ell_1!^\sigma (1+2C_{0,b})^{\ell_1}  \sup_{(t,\xi)\in\g}e^{L|\xi|^{\frac{1}{\mu}}}|\partial_\xi^{\ell'_3}\partial_t^{\ell_2}\widehat{f}(t,\xi)|\\
    &\leq \varepsilon^{-1} \sum_{|\ell'|=\beta} \sum_{|\ell|=\gamma}\binom{\beta}{\ell'}\binom{\gamma}{\ell}K_{c_0,q_0}^{\ell_1'}
    (4\pi C_{b,a_0})^{\ell_2'}C_{1,b}^{\ell_1} (1+2C_{0,b})^{\ell_1}\textstyle{(\frac{2\mu}{L})^{\mu\ell_1}}\ell_1'!\ell_1!^\mu\ell_1!^\sigma C_{0,f}C_{1,f}^{\ell_3'+\ell_2}\ell_2!^{\sigma}\ell_3'!^\mu\\
    &\leq C_0 \tilde{C}_{1}^{\beta+\gamma} 3^\beta 2^\gamma \beta!^{\mu}\gamma!^{\sigma+\mu}\leq C_0C_1^{\beta+\gamma}\beta!^\mu\gamma!^{\sigma+\mu},
\end{align*}
for every $\xi<0$, where $C_0=\varepsilon^{-1}C_{0,f}$, $\tilde{C}_1=K_{c_0,q_0}\max\{4\pi C_{b,a_0},1\}(1+2C_{0,b})(\frac{2\mu}{L})^\mu C_{1,b}C_{1,f},$ $C_1=6\tilde{C}_1$, and $K_{c_0,q_0}= \max\{4\pi|c_0|e\left(1+\varepsilon^{-1}\left(1+e^{-2\pi  \Real(q_0))}\right)\right),1\}$, .

For $\xi\geq0$, we obtain similar inequalities for $\widehat{u}(t,\xi)$ by an analogous argument. 

This proves that $\widehat{u}\in \mathcal{S}_{\sigma+\mu,\mu}(\g)$.  Consequently $u\in\mathcal{S}_{\sigma+\mu,\mu}\subset  \Smu$ by Corollary \ref{coro_gelf_transf} and therefore $P$ is $\mathscr{F}_{\mu}$-globally hypoelliptic. In the case where $b(t)\leq 0$ and so $b_0<0$, the proof is similar, only now we swap the arguments for $\widehat{u}(t,\xi)$ with respect to the cases $\xi\geq 0$ and $\xi<0$.
\end{proof}

\begin{corollary}
    Suppose that $b\not\equiv0$ does not change sign and $\frac{a_0}{b_0}\Real(q_0)+\Imag(q_0)\not\in\Z$. Then $P$ is $\mathscr{F}_\mu$-globally hypoelliptic for every $\mu\geq 1/2$.
\end{corollary}

\section{Considerations about global solvability}

In this section, we apply the topological properties proved in Section \ref{sec_prelim} to show that the two versions of global hypoellipticity considered in this work imply suitable notions of global solvability on the spaces $\mathcal{S}_{\sigma,\mu}(\G)$ and $\mathscr{F}_\mu(\G)$. 

\begin{definition}
Fix $\sigma\geq 1$ and $\mu\geq 1/2$.
    \begin{enumerate}
        \item A linear partial differential operator $P:\S\to\S$ is $\mathcal{S}_{\sigma,\mu}$-globally solvable if, for every $f\in(\ker\Pt)^\perp$, there exists $u\in\S$ such that $Pu=f$.

        \item A linear partial differential operator $P:\mathscr{F}_\mu(\G)\to \mathscr{F}_\mu(\G)$ is $\mathscr{F}_{\mu}$-globally solvable if, for every $f\in(\ker\Pt)^\perp$, there exists $u\in\mathscr{F}_\mu(\G)$ such that $Pu=f$.
    \end{enumerate}    
\end{definition}

These are the definitions of solvability considered in \cite{AvCap2025} and \cite{AvCapKir2024}. Since $\S$ and $\mathscr{F}_\mu(\G)$ are DFS spaces, the two notions of global solvability above are equivalent to the closedness of the range of the operator $P$ on the respective functional settings by \cite[Lemma 2.2]{Araujo2017}.

\begin{prop}
    Let $P:\S\to\S$ be a linear partial differential operator. If $P$ is $\mathcal{S}_{\sigma,\mu}$-globally hypoelliptic, then $P$ is $\mathcal{S}_{\sigma,\mu}$-globally solvable. Similarly, if $P:\mathscr{F}_\mu(\G)\to \mathscr{F}_\mu(\G)$ is $\mathcal{S}_{\mu}$-globally hypoelliptic, then $P$ is $\mathscr{F}_{\mu}$-globally solvable.
\end{prop}
\begin{proof}
    Since $\S$ is a DFS space, it follows from \cite[Corollary 5.4]{AFR22proc} that, if for some $\sigma_+>\sigma$ we have that
    \[u\in\mathcal{S}_{\sigma_+,\mu}(\G),\ Pu\in\S\ \Rightarrow\ \exists v\in \mathcal{S}_{\sigma,\mu}(\G)\text{ such that }Pu=Pv,\]
    then $P$ is $\mathcal{S}_{\sigma,\mu}$-globally solvable. In particular, the $\mathcal{S}_{\sigma,\mu}$-global hypoellipticity of $P$ implies the property above, which gives us the $\mathcal{S}_{\sigma,\mu}$-global solvability of $P$. 
    
    Next, note that given $\mu_+>\mu$, we have that $\mathscr{F}_\mu(\G)\subset \mathscr{F}_{\mu_+}(\G)$. The $\mathscr{F}_\mu$-global hypoellipticity implies the following property:
    \[u\in\mathscr{F}_{\mu_+}(\G),\ Pu\in\mathscr{F}_\mu(\G)\ \Rightarrow\ \exists v\in \mathscr{F}_\mu(\G)\text{ such that }Pu=Pv.\]
    
  Since $\mathscr{F}_\mu(\G)$ is a DFS space, again by \cite[Corollary 5.4]{AFR22proc} we obtain the $\mathscr{F}_\mu$-global solvability of $P$ .
\end{proof}

Hence, the results of the previous sections also give us sufficient conditions for the global solvability of operators on $\T^1\times\R$.

\appendix\section{Technical results}\label{appendix}

\begin{lemma}[Fa\`a di Bruno's Formula]\label{faa}
	Given $f\in C^\infty(\mathbb{R})$ and $N\in\mathbb{N}$, we have that
	\begin{equation*}{\partial_t^N}e^{f(t)}=e^{f(t)}\displaystyle\sum_{\tau\in\Delta(N)}\dfrac{N!}{\tau!}\prod_{\ell=1}^{N}\left(\dfrac{1}{\ell!}\partial_t^\ell f(t)\right)^{\tau_\ell},
    \end{equation*}
    where
    \begin{equation}\label{Delta_set_defi}
    \Delta(N)=\{\tau=(\tau_1,\dots,\tau_N)\in\N_0^{N}:\tau_1+2\tau_2+\dots+N\tau_N=N\}.
\end{equation}
\end{lemma}

\begin{lemma}\label{derivative-reciprocal}
    Let $g\in C^{\infty}(\R)$ be such that $g(t)\neq 0$ for every $t\in\R$. Then for every $n\in\N$ we have that
    \begin{equation*}
        \partial_t^n [g(t)^{-1}]=\sum_{k=1}^n(-1)^k\binom{n+1}{k+1}\frac{1}{g(t)^{k+1}}\partial_t^n[g(t)^k],\quad \forall t\in\R.
    \end{equation*}
\end{lemma}
\begin{proof}
	This follows from the more general version of the Faà di Bruno's Formula. Alternatively, another proof can be found in  \cite{Leslie01101991}.
\end{proof}

\begin{lemma}\label{lemma_algebraic_1}
    If $\tau=(\tau_1,\dots,\tau_N)\in\Delta(N)$ then, for any $\sigma\geq 1$, we have that
    \begin{equation*}
        |\tau|!^{\sigma}\prod_{\ell=1}^{N}\ell!^{(\sigma-1)\tau_{\ell}}\leq |\tau|!N!^{\sigma-1}
    \end{equation*}
\end{lemma}
\begin{proof}
    \cite[Lemma 2.1]{BDG2018}.
\end{proof}

\begin{lemma}\label{lemma_algebraic_2}
    For each $\N\in\N$, $\tau\in\N_0^N$ and $R\in\R$ we have that
    \begin{equation*}
        \sum_{\Delta(N)}\frac{|\tau|!}{\tau!}R^{|\tau|}=R(1+R)^{N-1}.
    \end{equation*}
\end{lemma}
\begin{proof}
	\cite[Lemma 1.4.1]{krantz}.
\end{proof}

\begin{lemma}\label{lemma-exponential}
Let $L,\mu>0$ and $s\in\N_0$. Then
\begin{equation*}
     e^{-L|x|^{\frac{1}{\mu}}}|x|^{s}\leq (\mu/L)^{\mu s}s!^{\mu},\quad\forall x\in\R.
\end{equation*}
\end{lemma}
\begin{proof}
    The proof is a consequence of the the estimate $e^{-t}t^N\leq N!$, for every $t>0$ and $N\in\mathbb{N}_0$. Indeed, let $N=s$ and $t=\frac{L}{\mu}|x|^{\frac{1}{\mu}}$. Then, we have
    \[s!\geq e^{-\frac{L}{\mu}|x|^{\frac{1}{\mu}}}\left(\frac{L}{\mu}|x|^{\frac{1}{\mu}}\right)^{s} = \left(e^{-L|x|^{\frac{1}{\mu}}}\right)^{\frac{1}{\mu}}\left(\frac{L}{\mu}|x|^{\frac{1}{\mu}}\right)^{s}. \]
    Raising both sides of the inequality above by $\mu$, we obtain the desired estimate.
\end{proof}

\begin{lemma}\label{lemmaintegralineq}
	Let $\psi\in C^{\infty}(\T^1)$ be a smooth real-valued function such that $\psi\geq 0$ and $s_0\in\T^1$ is a zero of order greater than one for $\psi$, that is, $\psi(s_0)=0=\psi'(s_0)$. Then, there exists $M>0$ such that for all $\lambda>0$ sufficiently big and $\delta>0$ we have that
	$$\int_{s_0-\delta}^{s_0+\delta}e^{-\lambda\psi(s)}\,\mathrm{d}s\geq \left(\int_{-\delta}^{\delta}e^{-s^2}\,\mathrm{d}s\right)\lambda^{-1/2}M^{-1/2}.$$
\end{lemma}
\begin{proof}
	See \cite{KowTmRn}.
\end{proof}

\begin{lemma}\label{lemmaodesol}
	Let $g,\theta \in C^{\infty}(\T^1)$, and set $\theta_0 = \frac{1}{2\pi}\int_0^{2\pi}\theta(t)\,\mathrm{d}t$. If $\theta_0\not\in i\mathbb{Z}$, then the differential equation
	\begin{equation}\label{ode}
		\partial_t u(t)+\theta(t)u(t)=g(t),\,\qquad t\in\T^1,
	\end{equation}
	admits a unique solution in $C^\infty(\T^1)$ given by
	\begin{equation}\label{sol-}
		u(t) = \frac{1}{1-e^{-2\pi\theta_0}}\int_0^{2\pi}g(t-s)e^{-\int_{t-s}^t\theta(\tau)\,\mathrm{d}\tau}\,\mathrm{d}s,
	\end{equation} 
	or equivalently by
	\begin{equation}\label{sol+}
		u(t) = \frac{1}{e^{2\pi\theta_0}-1}\int_0^{2\pi}g(t+s)e^{\int_{t}^{t+s}\theta(\tau)\,\mathrm{d}\tau}\,\mathrm{d}s.
	\end{equation}
	
	If $\theta_0\in i\mathbb{Z}$, then equation \eqref{ode} admits infinitely many solutions given by
	\begin{equation*}
		u_\lambda(t) = \lambda e^{-\int_0^t\theta(\tau)\,\mathrm{d}\tau}+\int_0^t g(s)e^{-\int_s^t\theta(\tau)\,\mathrm{d}\tau}\,\mathrm{d}s,
	\end{equation*}
	for every $\lambda\in\mathbb{R}$, if and only if
	$$ \int_0^{2\pi}g(t)e^{\int_0^t\theta(\tau)\,\mathrm{d}\tau}\,\mathrm{d}t=0.$$
\end{lemma}
\begin{proof}
   \cite[Lemma A.4]{Lucy}.
\end{proof}

\bibliographystyle{plain}
\bibliography{references}

\begin{thebibliography}{10}

\bibitem{Araujo2017}
G.~Ara{\'u}jo.
\newblock Regularity and solvability of linear differential operators in
  {Gevrey} spaces.
\newblock {\em Math. Nachr.}, 291(5-6):729--758, 2018.

\bibitem{AFR22proc}
G.~Ara{\'u}jo, I.~A. Ferra, and L.~F. Ragognette.
\newblock Global analytic hypoellipticity and solvability of certain operators
  subject to group actions.
\newblock {\em Proc. Am. Math. Soc.}, 150(11):4771--4783, 2022.

\bibitem{AAC2022}
A.~Arias~Junior, A.~Ascanelli, and M.~Cappiello.
\newblock The {Cauchy} problem for 3-evolution equations with data in
  {Gelfand}-{Shilov} spaces.
\newblock {\em J. Evol. Equ.}, 22(2):40, 2022.

\bibitem{AKM19}
A.~Arias~Junior, A.~Kirilov, and C.~de~Medeira.
\newblock Global {G}evrey hypoellipticity on the torus for a class of systems
  of complex vector fields.
\newblock {\em J. Math. Anal. Appl.}, 474(1):712--732, 2019.

\bibitem{AC2019}
A.~Ascanelli and M.~Cappiello.
\newblock Schr{\"o}dinger-type equations in {Gelfand}-{Shilov} spaces.
\newblock {\em J. Math. Pures Appl.}, 132:207--250, 2019.

\bibitem{Ber1994}
A.~P. Bergamasco.
\newblock Perturbations of globally hypoelliptic operators.
\newblock {\em J. Differ. Equations}, 114(2):513--526, 1994.

\bibitem{Ber1999}
A.~P. Bergamasco.
\newblock Remarks about global analytic hypoellipticity.
\newblock {\em Trans. Am. Math. Soc.}, 351(10):4113--4126, 1999.

\bibitem{BCM1993}
A.~P. Bergamasco, P.~D. Cordaro, and P.~A. Malagutti.
\newblock Globally hypoelliptic systems of vector fields.
\newblock {\em J. Funct. Anal.}, 114(2):267--285, 1993.

\bibitem{BDG2018}
A.~P. Bergamasco, P.~L. Dattori~da Silva, and R.~B. Gonzalez.
\newblock Global solvability and global hypoellipticity in {Gevrey} classes for
  vector fields on the torus.
\newblock {\em J. Differ. Equations}, 264(5):3500--3526, 2018.

\bibitem{CGR2011}
M.~Cappiello, T.~Gramchev, and L.~Rodino.
\newblock Exponential estimates and holomorphic extensions for semilinear
  elliptic pseudodifferential equations.
\newblock {\em Complex Var. Elliptic Equ.}, 56(12):1129--1142, 2011.

\bibitem{AvCap2022}
F.~de~{\'A}vila~Silva and M.~Cappiello.
\newblock Time-periodic {Gelfand}-{Shilov} spaces and global hypoellipticity on
  {{\(\mathbb{T}\times\mathbb{R}^n\)}}.
\newblock {\em J. Funct. Anal.}, 282(9):29, 2022.

\bibitem{AvCap2025}
F.~de~{\'A}vila~Silva and M.~Cappiello.
\newblock Globally solvable time-periodic evolution equations in
  {Gelfand}-{Shilov} classes.
\newblock {\em Math. Ann.}, 391(1):399--430, 2025.

\bibitem{AvCapKir2024}
F.~de~{\'A}vila~Silva, M.~Cappiello, and A.~Kirilov.
\newblock Systems of differential operators in time-periodic {Gelfand}-{Shilov}
  spaces.
\newblock {\em Ann. Mat. Pura Appl.}, 204(2):643--665, 2025.

\bibitem{AGK2018}
F.~de~{\'A}vila~Silva, T.~Gramchev, and A.~Kirilov.
\newblock Global hypoellipticity for first-order operators on closed smooth
  manifolds.
\newblock {\em J. Anal. Math.}, 135(2):527--573, 2018.

\bibitem{GelfandShilov}
I.~M. Gelfand and G.~E. Shilov.
\newblock {\em Generalized functions. {Vol}. 2: {Spaces} of fundamental and
  generalized functions}.
\newblock {New} {York}-{London}: {Academic} {Press}, 1968.

\bibitem{GPR2011}
T.~Gramchev, S.~Pilipovi\'c, and L.~Rodino.
\newblock Eigenfunction expansions in {{\(\mathbb R^{n}\)}}.
\newblock {\em Proc. Am. Math. Soc.}, 139(12):4361--4368, 2011.

\bibitem{GW}
S.~J. Greenfield and N.~R. Wallach.
\newblock Global hypoellipticity and {Liouville} numbers.
\newblock {\em Proc. Am. Math. Soc.}, 31:112--114, 1972.

\bibitem{Hounie}
J.~Hounie.
\newblock Globally hypoelliptic and globally solvable first order evolution
  equations.
\newblock {\em Trans. Am. Math. Soc.}, 252:233--248, 1979.

\bibitem{KMR2021}
A.~Kirilov, W.~A.~A. de~Moraes, and M.~Ruzhansky.
\newblock Global hypoellipticity and global solvability for vector fields on
  compact {Lie} groups.
\newblock {\em J. Funct. Anal.}, 280(2):39, 2021.

\bibitem{Komatsu1967}
H.~Komatsu.
\newblock Projective and injective e limits of weakly compact sequences of
  locally convex spaces.
\newblock {\em J. Math. Soc. Japan}, 19:366--383, 1967.

\bibitem{KowTmRn}
A.~P. Kowacs.
\newblock Fourier {Analysis} on $\mathbb{T}^m\times\mathbb{R}^n$ and
  {Applications} to {Global} {Hypoellipticity}.
\newblock Preprint, {arXiv}:2306.15578, 2023.

\bibitem{Kow_chap}
A.~P. Kowacs.
\newblock Fourier analysis on hypercylinders.
\newblock In D.~Cardona and B.~Grajales, editors, {\em Analysis and PDE in
  Latin America}, pages 114--121, Cham, 2024. Springer Nature Switzerland.

\bibitem{krantz}
S.~Krantz and H.~R. Parks.
\newblock {\em A primer of real analytic functions}, volume~4 of {\em Basler
  Lehrb{\"u}ch.}
\newblock Basel: Birkh{\"a}user Verlag, 1992.

\bibitem{zbMATH05257541}
T.~Krasi{\'n}ski.
\newblock On the {{\L}}ojasiewicz exponent at infinity of polynomial mappings.
\newblock {\em Acta Math. Vietnam.}, 32(2-3):189--203, 2007.

\bibitem{Leslie01101991}
R.~A. Leslie.
\newblock How not to repeatedly differentiate a reciprocal.
\newblock {\em The American Mathematical Monthly}, 98(8):732--735, 1991.

\bibitem{NicRod}
F.~Nicola and L.~Rodino.
\newblock {\em Global pseudo-differential calculus on {Euclidean} spaces},
  volume~4 of {\em Pseudo-Differ. Oper., Theory Appl.}
\newblock Basel: Birkh{\"a}user, 2010.

\bibitem{Petersson}
A.~Petersson.
\newblock Fourier characterizations and non-triviality of {Gelfand}-{Shilov}
  spaces, with applications to {Toeplitz} operators.
\newblock {\em J. Fourier Anal. Appl.}, 29(3):24, 2023.

\bibitem{Rod_Gevrey}
L.~Rodino.
\newblock {\em Linear partial differential operators in {Gevrey} spaces}.
\newblock Singapore: World Scientific, 1993.

\bibitem{Lucy}
L.~T. Takahashi.
\newblock Hipoeliticidade global de certas classes de operadores diferenciais
  parciais.
\newblock Master's {T}hesis, Universidade Federal de São Carlos, São Carlos,
  Brazil, 1995.

\bibitem{Treves_TVS}
F.~Tr\`eves.
\newblock {\em {T}opological {V}ector {S}paces, {D}istributions and {K}ernels}.
\newblock Pure and Applied Mathematics. Academic Press, 1967.

\bibitem{zbMATH05789579}
H.~H. Vui and N.~H. Duc.
\newblock {{\L}}ojasiewicz inequality at infinity for polynomials in two real
  variables.
\newblock {\em Math. Z.}, 266(2):243--264, 2010.

\bibitem{Webb}
J.~H. Webb.
\newblock Sequential convergence in locally convex spaces.
\newblock {\em Proc. Camb. Philos. Soc.}, 64:341--364, 1968.

\end{thebibliography}

\end{document}